\documentclass[12pt,reqno,a4paper]{amsart}

\usepackage[margin=25truemm]{geometry}
\usepackage{amssymb}
\usepackage{amsmath}
\usepackage{mathtools}
\usepackage{thmtools}
\usepackage{amscd}
\usepackage{amsthm}
\usepackage{amsbsy}
\usepackage{ascmac}
\usepackage{stmaryrd}
\usepackage{amsfonts}
\usepackage{mathrsfs}
\usepackage{bm}
\usepackage{bbm}
\usepackage{xcolor}
\usepackage{tikz}
\hypersetup{% option list for hyperref
 setpagesize=false,
 hypertexnames=false,
 hyperfootnotes=true,
 bookmarksnumbered=true,%
 bookmarksopen=true,%
 colorlinks=true,%
 linkcolor=blue,
 citecolor=blue,
}
\theoremstyle{plain}
	\newtheorem{theorem}{Theorem}[section]
	\newtheorem{lemma}[theorem]{Lemma}
	\newtheorem{proposition}[theorem]{Proposition}
	\newtheorem{corollary}[theorem]{Corollary}
	\newtheorem*{theorem*}{Theorem}
\theoremstyle{definition}
	\newtheorem{definition}{Definition}[section]
	\newtheorem{remark}{Remark}[section]

\theoremstyle{remark}

\usepackage[shortlabels]{enumitem}
\allowdisplaybreaks[4]
\numberwithin{equation}{section}
\newcommand{\even}{\mathrm{even}}

\newcommand{\rad}{\mathrm{rad}}

\DeclareMathOperator{\supp}{supp}
\DeclareMathOperator{\spn}{span}

\DeclareMathOperator{\re}{Re}

\DeclareMathOperator{\im}{Im}
\usepackage{comment}
\begin{document}

\title[Threshold even sol.s to ($\delta$NLS) at low freq.]{Scattering and Blow-up for threshold even solutions to the nonlinear Schr\"{o}dinger equation with repulsive delta potential at low frequencies}

\author[S. Gustafson]{Stephen Gustafson}
\address[S. Gustafson]{University of British Columbia, 1984 Mathematics Rd., Vancouver, Canada V6T1Z2.}
\email{gustaf@math.ubc.ca}

\author[T. Inui]{Takahisa Inui}
\address[T. Inui]{Department of Mathematics, Graduate School of Science, Osaka University, Toyonaka, Osaka, Japan 560-0043.}
%\newline 
%University of British Columbia, 1984 Mathematics Rd., Vancouver, Canada V6T1Z2.}
\email{inui@math.sci.osaka-u.ac.jp}

\date{\today}


\maketitle

\begin{abstract}
We consider the $L^2$-supercritical nonlinear Schr\"{o}dinger equation with a repulsive Dirac delta potential in one dimensional space. 
In a previous work, we clarified the global dynamics of even solutions with the same action as the high-frequency ground state standing wave solutions. In that case, there are obvious non-scattering global solutions, i.e., the standing waves. In this paper, we show a scattering and blow-up dichotomy for threshold even solutions in the low-frequency case. We emphasize that this dichotomy still holds at the critical frequency between high and low. 
\end{abstract}

\tableofcontents

%%%%%%%%%%%%%%%%%%%%%%%%%%%%%%%%%%%%%%%%%%
%%%%%%%%%%%%%%%%%%%%%%%%%%%%%%%%%%%%%%%%%%
\section{Introduction}

\subsection{Motivation}

We consider the Schr\"{o}dinger equation with a repulsive Dirac delta potential and a focusing power type nonlinearity in one dimension: 
\begin{align}
\label{deltaNLS}
\tag{$\delta$NLS}
	\begin{cases}
	i\partial_t u +\partial_x^2 u +\gamma \delta u +|u|^{p-1}u=0, &(t,x) \in I \times \mathbb{R},
	\\
	u(0,x)=u_0(x) & x \in \mathbb{R},
	\end{cases}
\end{align}
where $\gamma<0$, $ \delta$ is the Dirac delta with the mass at the origin, and $p>5$. Though the equation is not scaling invariant, the condition $p>5$ means that the nonlinearity is $L^2$-supercritical. The Schr\"{o}dinger operator $-\Delta_\gamma := -\partial_x^2 -\gamma \delta_0$ is defined by
\begin{align*}
	-\Delta_\gamma f &:= -  f'' \text{ for } f\in \mathcal{D}(-\Delta_\gamma),
	\\
	 \mathcal{D}(-\Delta_\gamma) &:= \{f \in H^1(\mathbb{R}) \cap H^2(\mathbb{R}\setminus \{0\}): f'(0+) - f'(0-) = -\gamma f(0)\},
\end{align*}
where $f'$ denotes the spatial derivative of $f$. 
The condition $\gamma<0$ means that the potential is repulsive. In the repulsive case, it is known that the operator $-\Delta_\gamma$ is non-negative and self-adjoint on $L^2(\mathbb{R})$ (see \cite[Chapter I.3]{AGHKH05}) and thus the linear Schr\"{o}dinger propagator $e^{it\Delta_\gamma}$ is unitary on $L^2(\mathbb{R})$ by the Stone theorem. We may also define the operator $-\Delta_\gamma$ through the quadratic form 
\begin{align*}
	q( f , g) := \int_{\mathbb{R}} f'(x)\overline{g'(x)}dx -\gamma f(0)\overline{g(0)}
\end{align*}
for $f,g \in H^1(\mathbb{R})$ (see \cite{Kat95} and \cite[Example 10.7]{Sch12}). It is well-known that the equation \eqref{deltaNLS} is locally well-posed in $H^1(\mathbb{R})$, and the mass and the energy, defined respectively by 
\begin{align*}
	M(f) &:= \|f\|_{L^2}^2,
	\\
	E_\gamma (f) &:= \frac{1}{2}\|f'\|_{L^2}^2 -\frac{\gamma}{2}|f(0)|^2 - \frac{1}{p+1}\|f\|_{L^{p+1}}^{p+1},
\end{align*}
are conserved.
See \cite[Section 3]{GHW04}, \cite[Theorem 3.7.1]{Caz03}, and \cite[Proposition 1]{FOO08} for the local well-posedness. 

Global dynamics of nonlinear Schr\"{o}dinger equations has attracted a lot of interest after the pioneering work by Kenig and Merle \cite{KeMe06}. Global dynamics below the ground state for the Schr\"{o}dinger equation with mass supercritical and energy subcritical nonlinearity, but without potential, are well-studied (see \cite{HoRo08,DHR08,HoRo10,FXC11,AkNa13,Gue14,Guo16}), as are the global dynamics at the ground state level (see \cite{DuRo10,CFR22}). See also \cite{NaSc12} for the global dynamics above the ground state. 
More recently, dynamics of nonlinear Schr\"{o}dinger equations with potential have been considered: see \cite{HaIk20,Din21,Hon16} for general potentials and \cite{ZhZh14,KMVZ17,LMM18} for the inverse square potential. 
The authors have considered the global dynamics of the equation \eqref{deltaNLS}. To explain the known results, we first give some notations and definitions. 
For a frequency $\omega>0$, we define the action $S_{\omega,\gamma}$ by
\begin{align*}
	S_{\omega,\gamma} (f):=E_\gamma (f) + \frac{\omega}{2} M(f)
\end{align*}
and the virial functional $K_\gamma$ by 
\begin{align*}
	K_\gamma(f):=  \|f'\|_{L^2}^2 - \frac{\gamma}{2} |f(0)|^2 - \frac{p-1}{2(p+1)} \|f\|_{L^{p+1}}^{p+1}.
\end{align*}
The scaling critical regularity index is 
\[
  s_c := \frac{1}{2} - \frac{2}{p-1}.
\]
\begin{definition}[scattering]
We say that a solution $u$ scatters in the positive (resp. negative) time direction if the solution exists at least on $[0,\infty)$ (resp. $(-\infty,0]$) and there exists $u_{+} \in H^{1}(\mathbb{R})$  (resp. $u_{-} \in H^{1}(\mathbb{R})$) such that 
\begin{align*}
	\|u(t) - e^{it\Delta_{\gamma}} u_{+(-)}\|_{H^{1}} \to 0 \text{ as } t \to \infty \text{ (resp. $-\infty$)}.
\end{align*}
We only say that a solution $u$ scatters if the solution scatters in both time directions. 
\end{definition}
\begin{remark}
\begin{enumerate}[(1)]
\item 
Mizutani \cite{Miz20} shows that, if a solution scatters in the positive time direction, then there exists $\widetilde{u}_{+} \in H^1(\mathbb{R})$ such that 
\begin{align*}
	\|u(t) - e^{it\partial_{x}^{2}} \widetilde{u}_{+}\|_{H^{1}} \to 0 \text{ as } t \to \infty.
\end{align*}
This means the solution of the nonlinear equation behaves freely at time infinity. 
\item For the defocusing nonlinear Schr\"{o}dinger equation with repulsive Dirac equation, Banica and Visciglia \cite{BaVi16} show that all solutions scatter. 
\end{enumerate}
\end{remark}
\begin{definition}[blow-up, grow-up]
\begin{enumerate}[(1)]
\item We say that a solution $u$ blows up in positive time if the solution exists until finite $T>0$ and does not extend beyond $T$. By the blow-up alternative, if $u$ blows up in positive time, then $\|u'(t)\|_{L^2} \to \infty$ as $t$ goes to maximal positive existence time. 
\item We say that a solution $u$ grows up in positive time if it exists on $[0,\infty)$ and 
$\limsup_{t\to \infty}\|u'(t)\|_{L^{2}}=\infty$. 
\end{enumerate}
We define blow-up (resp. grow-up) in negative time in a similar way. 
We say that a solution $u$ blows up if it blows up in both time directions. 
We say a solution blows up or grows up if the solution blows or grows up in positive time and blows or grows up in negative time.
\end{definition}

To study the global dynamics, minimization problems related to the ground state play an important role. We consider the following two:
\begin{align*}
	n_{\omega,\gamma}&:=\inf\{ S_{\omega,\gamma}(f): f \in H^{1}(\mathbb{R})\setminus\{0\}, K_{\gamma}(f)=0\},
	\\
	r_{\omega,\gamma}&:=\inf\{ S_{\omega,\gamma}(f): f \in H_{\even}^{1}(\mathbb{R})\setminus\{0\}, K_{\gamma}(f)=0\}, 
\end{align*}
where $H_{\even}^{1}(\mathbb{R})$ denotes the even functions in $H^{1}(\mathbb{R})$. We have: 

\begin{proposition}[\cite{FuJe08,IkIn17}]
\label{prop1.1}
Let $\gamma<0$.
\begin{enumerate}
\item For all $\omega>0$, $n_{\omega,\gamma}=n_{\omega,0}$ and $n_{\omega,\gamma}$ is not attained.
\item $n_{\omega,\gamma} < r_{\omega,\gamma}$ and 
\begin{align*}
	r_{\omega,\gamma}
	\begin{cases}
	= 2 n_{\omega,0} & \text{ if } 0< \omega \leq \frac{\gamma^{2}}{4},
	\\
	< 2 n_{\omega,0} & \text{ if } \omega > \frac{\gamma^{2}}{4}.
	\end{cases}
\end{align*}
\item If $0 < \omega \leq \gamma^{2}/4$, then $r_{\omega,\gamma}$ is not attained. 
If $\omega> \gamma^{2}/4$, then $r_{\omega,\gamma}$ is attained by 
\begin{align*}
	Q_{\omega,\gamma}(x):= c_p \omega^{\frac{1}{p-1}} \left[ 2 \cosh\left\{ \frac{p-1}{2}\sqrt{\omega}|x| + \tanh^{-1}\left( \frac{\gamma}{2\sqrt{\omega}} \right) \right\} \right]^{-\frac{2}{p-1}},
\end{align*}
where $c_p:=\left\{ 2(p+1)\right\}^{\frac{1}{p-1}}$. 
The function $Q_{\omega,\gamma}$ is the unique (up to a complex phase) solution of 
\begin{align}
\label{eleq}
	-Q_{\omega,\gamma}'' -\gamma \delta Q_{\omega,\gamma} + \omega Q_{\omega,\gamma} =|Q_{\omega,\gamma}|^{p-1}Q_{\omega,\gamma}.
\end{align}
\end{enumerate}
\end{proposition}

\begin{remark}
\label{rmk1.2}
Let
\begin{align*}
	Q(x):= c_p\left\{ 2 \cosh\left( \frac{p-1}{2}x\right) \right\}^{-\frac{2}{p-1}}, 
\end{align*} 
known as the ground state of the equation 
\begin{align}
\label{ellip}
	-Q'' + Q =Q^p.
\end{align}
Then $Q_{\omega,\gamma}$ is given by a translation, reflection, and scaling as follows: 
\begin{align*}
	Q_{\omega,\gamma}(x):= \omega^{\frac{1}{p-1}} Q\left(\sqrt{\omega}|x| + \frac{2}{p-1} \tanh^{-1}\left( \frac{\gamma}{2\sqrt{\omega}} \right) \right)
\end{align*}
for $\omega>\gamma^2/4$. 
\end{remark}

\begin{remark}
Fukuizumi and Jeanjean \cite{FuJe08} show the instability of the ground state standing waves if $p\geq 5$. Le Coz et.al \cite{LCFFKS08} obtained the strong instability of the standing waves. 
In their papers \cite{FuJe08} and  \cite{LCFFKS08}, the stability of the ground states in the case $1<p<5$ is also investigated (see also \cite{Oht11} for the degenerate case). For the attractive Dirac delta potential ($\gamma>0$), the orbital stability of the standing waves is studied by Fukuizumi, Ohta, and Ozawa \cite{FOO08} (see also \cite{GHW04}) in the non-degenerate case and Ohta \cite{Oht11} in the degenerate case. See Ohta and Yamaguchi \cite{OhYa16} and Fukaya and Ohta \cite{FuOh19} for the strong instability (see also \cite{Oht19} for a review). Recently, Masaki, Murphy, and Segata \cite{MMS23} proved the asymptotic stability of small solitary waves for $p > 5$ under a suitable spectral condition. See also \cite{KaOh09,OhYa16}. 
\end{remark}

The function $e^{i\omega t}Q_{\omega,\gamma}$ is a non-scattering global solution to \eqref{deltaNLS} if $\omega>\gamma^{2}/4$. In the non-even case, $n_{\omega,\gamma}=n_{\omega,0} = S_{\omega,0}(Q_{\omega,0})$ gives a threshold for a dichotomy result even though $e^{i\omega t}Q_{\omega,0}$ is not a solution to \eqref{deltaNLS}:

\begin{theorem}[Global dynamics below $Q_{\omega,0}$ (Ikeda--Inui \cite{IkIn17})]
For $\omega>0$, we define the potential well set $PW_{\omega,\gamma}:= \{f\in H^1(\mathbb{R}):S_{\omega,\gamma}(f) < n_{\omega,\gamma} \}$. Assume that $u_0 \in \cup_{\omega>0}PW_{\omega,\gamma}$. 
%Let $\omega>0$. Assume that $S_{\omega,\gamma}(u_{0}) < n_{\omega,\gamma}(=n_{\omega,0})$. 
Then the following hold for \eqref{deltaNLS}:
\begin{enumerate}
\item If $K_\gamma(u_{0})\geq0$, the solution scatters.
\item If $K_\gamma(u_{0})<0$, the solution blows up or grows up.
\end{enumerate}
\label{belowresult}
\end{theorem}
\begin{remark}
\begin{enumerate}[(1)]
\item 
Here (and in Theorem \ref{thm1.5} below) $K_\gamma(u_{0})=0$ implies $u_0=0$.
\item
By scaling, we can re-express Theorem~\ref{belowresult} by replacing the assumption $u_0 \in \cup_{\omega>0}PW_{\omega,\gamma}$ %on the action
with the mass-energy condition $E_\gamma(u_0)M(u_0)^{\frac{1-s_c}{s_c}} < E(Q) M(Q)^{\frac{1-s_c}{s_c}}$. 
\end{enumerate}
\end{remark}
We also have a dichotomy result at the threshold:
%\begin{theorem}[Global dynamics on the threshold (Ardila--Inui \cite{ArIn22}, Inui \cite{Inu23p})]
%Let $\omega>0$. Assume that $S_{\omega,\gamma}(u_{0}) = n_{\omega,\gamma}(=n_{\omega,0})$. Then the following hold for \eqref{deltaNLS}.
%\begin{enumerate}
%\item If $K_\gamma(u_{0})>0$, the solution scatters.
%\item If $K_\gamma(u_{0})<0$ and $xu_{0}(x) \in L^2(\mathbb{R})$, the solution blows up. 
%\end{enumerate}
%\end{theorem}
%
%
%This is also rewritten in the similar way to the previous statement.

\begin{theorem}[Global dynamics on the threshold (Ardila--Inui \cite{ArIn22}, Inui \cite{Inu23p})]
Assume that $u_0 \in H^1(\mathbb{R})$ satisfies $E_\gamma(u_0)M(u_0)^{\frac{1-s_c}{s_c}} = E(Q) M(Q)^{\frac{1-s_c}{s_c}}$. Then the following hold for \eqref{deltaNLS}.
\begin{enumerate}
\item If $K_\gamma(u_{0})>0$, the solution scatters.
\item If $K_\gamma(u_{0})<0$ and $xu_{0}(x) \in L^2(\mathbb{R})$, the solution blows up. 
\end{enumerate}
\label{onresult}
\end{theorem}
\begin{remark}
\begin{enumerate}[(1)]
\item $K_\gamma(u_{0})=0$ does not occur since there is no ground state on the threshold. 
\item The dichotomy results Theorems~\ref{belowresult} and~\ref{onresult} are optimal in this sense: for any $\epsilon > 0$, there exists a non-scattering, global moving one-soliton solution $u$ with $E_\gamma(u)M(u)^{\frac{1-s_c}{s_c}} < E(Q) M(Q)^{\frac{1-s_c}{s_c}} + \epsilon$ (see \cite{GIS23p}). 
\item In Theorems~\ref{belowresult} and~\ref{onresult}, the %The 
condition on the sign of $K_\gamma(u_0)$ at the initial time can be replaced by the sign of $K_0(u_0)$, or $\|Q_{1,0}\|_{L^2}^{1-s_c}\|Q_{1,0}\|_{\dot{H}^1}^{s_c}-\|u_0\|_{L^2}^{1-s_c}\|\sqrt{-\Delta_\gamma}u_0\|_{L^2}^{s_c}$,  or $\|Q_{1,0}\|_{L^2}^{1-s_c}\|Q_{1,0}\|_{\dot{H}^1}^{s_c}-\|u_0\|_{L^2}^{1-s_c}\|u_0\|_{\dot{H}^1}^{s_c}$. That is the solution scatters (e.g.) if one of them (therefore all of them) is  positive. See e.g. \cite{HIIS22p}. It is notable that functionals not involving the potential term also determine the global behavior. 
\end{enumerate}
\end{remark}

For even solutions, we have the following global dynamics result below $r_{\omega,\gamma}$, which is strictly larger than $n_{\omega,\gamma}$ by Proposition \ref{prop1.1}:
\begin{theorem}[Global dynamics of even solutions below $r_{\omega,\gamma}$ (Ikeda--Inui \cite{IkIn17})]
\label{thm1.5}
For $\omega>0$, we define the potential well set $PW_{\omega,\gamma}^\even:= \{f\in H_\even^1(\mathbb{R}):S_{\omega,\gamma}(f) < r_{\omega,\gamma} \}$. Assume that $u_0 \in \cup_{\omega>0}PW_{\omega,\gamma}^\even$. 
%Let $\omega>0$. Assume that $u_{0}$ is even and $S_{\omega,\gamma}(u_{0}) < r_{\omega,\gamma}$. Then the following hold for \eqref{deltaNLS}. 
\begin{enumerate}
\item If $K_\gamma(u_{0})\geq0$, the solution scatters.
\item If $K_\gamma(u_{0})<0$, the solution blows up or grows up. 
\end{enumerate}
\label{belowresulteven}
\end{theorem}

If $\omega>\gamma^{2}/4$, the ground state $Q_{\omega,\gamma}$ is a non-scattering, global, even solution on the
threshold. We also have solutions converging to the ground state as follows:
\begin{theorem}[Existence of special solutions (Gustafson--Inui \cite{GuIn22p})]
Let $\omega>\gamma^{2}/4$ be fixed. There exist two even solutions $U^{\pm}$ to \eqref{deltaNLS} such that 
\begin{itemize}
\item $M(U^{\pm})=M(Q_{\omega,\gamma})$, $E_{\gamma}(U^{\pm})=E_{\gamma}(Q_{\omega,\gamma})$, $U^{\pm}$ exist at least on $[0,\infty)$ and there exists $c>0$ such that 
\begin{align*}
	\|U^{\pm}(t)-e^{i\omega t}Q_{\omega,\gamma}\|_{H^{1}} \lesssim e^{-ct} \text{ for } t\geq 0.
\end{align*}
\item $K_{\gamma}(U^{+}(0))<0$ and $U^{+}$ blows up in finite negative time. 
\item $K_{\gamma}(U^{-}(0))>0$ and $U^{-}$ scatters backward in time.  
\end{itemize}
\end{theorem}

Moreover, we have a global dynamics result in the high frequency case:
\begin{theorem}[Global dynamics on the threshold (Gustafson--Inui \cite{GuIn22p})]
Let $\omega>\gamma^{2}/4$. Assume that $u_{0} \in H^1(\mathbb{R})$ is even and satisfies the mass-energy condition
\begin{align*}
	M(u_{0})=M(Q_{\omega,\gamma}) \text{ and } E_{\gamma}(u_{0})=E_{\gamma}(Q_{\omega,\gamma}).
\end{align*}
Then the following are true for \eqref{deltaNLS}. 
\begin{enumerate}
\item If $K_{\gamma}(u_{0})>0$, the solution $u$ scatters, or else $u=U^{-}$ up to symmetry.
\item If $K_{\gamma}(u_{0})=0$, then $u=e^{i\omega t}Q_{\omega,\gamma}$ up to symmetry.
\item If $K_{\gamma}(u_{0})<0$ and $\int_{\mathbb{R}}|xu_{0}(x)|^{2}dx<\infty$, the solution $u$ blows up, or else $u=U^{+}$ up to symmetry. 
\end{enumerate}
\label{highresult}
\end{theorem}

Our aim in the present paper is to classify the threshold dynamics of
even solutions in the low frequency case $0<\omega \leq \gamma^2/4$. Since there is no ground state, one can expect a scattering and blow-up dichotomy. 

%%%%%%%%%%%%%%%%%%%%%%%%%%%%%%%%%%%%%%%%%%
\subsection{Main result}

We show the following scattering and blow-up dichotomy result at the threshold in the low frequency case, for even solutions: 
\begin{theorem}
\label{thm}
Let $0<\omega \leq \gamma^2/4$. Assume that $u_0 \in H_\even^1(\mathbb{R})$ satisfies $E_{\gamma}(u_0) = 2 E(Q_{\omega,0})$ and $M(u_0)=2M(Q_{\omega,0})$. Then we have the following:
\begin{enumerate}[(1)]
\item If $K_\gamma(u_0)>0$, the solution $u$ scatters.
\item Assume in addition that $xu_0 \in L^2(\mathbb{R})$. If  $K_\gamma(u_0)<0$, the solution $u$ blows up. 
\end{enumerate}
\end{theorem}
\begin{remark}
\begin{enumerate}[(1)]
\item The mass-energy conditions $E_{\gamma}(u_0) = 2 E(Q_{\omega,0})$ and $M(u_0)=2M(Q_{\omega,0})$ come from the fact that $r_{\omega,\gamma}=2 n_{\omega,0}=2S_{\omega,0}(Q_{\omega,0})$ (see Proposition \ref{prop1.1}) if $0<\omega \leq \gamma^2/4$. 
By using the scaling property of $Q_{\omega,0}$, we can re-express Theorem~\ref{thm} by replacing these mass-energy conditions with the 
$\omega$-independent conditions $E_\gamma(u_0)M(u_0)^{\frac{1-s_c}{s_c}} = 2^{\frac{1}{s_c}}E(Q) M(Q)^{\frac{1-s_c}{s_c}}$ and $M(u_0) \geq 2M(Q_{\gamma^2/4,0})$, where
the additional mass condition ensures the low frequency. 
\item We note that $K_\gamma(u_{0})=0$ does not occur on the threshold in the low frequency case since there is no ground state.
\item The dichotomy results
Theorems~\ref{belowresulteven} and~\ref{thm} are optimal in this sense: for any $\epsilon > 0$, there exists a non-scattering, global, even, moving two-soliton solution $u$ with 
$E_\gamma(u_0)M(u_0)^{\frac{1-s_c}{s_c}} < 2^{\frac{1}{s_c}}E(Q) M(Q)^{\frac{1-s_c}{s_c}} + \epsilon$ (see~\cite{GIS23p}). 
\begin{comment}
Let $0<\omega \leq \gamma^2/4$ and $v>0$. Then setting 
\begin{align*}
	\mathcal{R}_\pm:= Q_{\omega,\gamma}(x\pm vt)e^{i\left( \pm\frac{1}{2}vx-\frac{1}{4}v^2t +\omega t \right)},
	\text{ and } 
	\mathcal{R}:= \mathcal{R}_+ + \mathcal{R}_-,
\end{align*}
then there exists a solution $u$ and $T\in \mathbb{R}$, $c,C>0$ such that
\begin{align*}
	\|u(t) -\mathcal{R}(t)\|_{H^1} \leq Ce^{-ct}
\end{align*}
for any $t>T$. 
This solution satisfies $E_\gamma(u_0)M(u_0)^{\frac{1-s_c}{s_c}} > 2^{\frac{1}{s_c}}E(Q) M(Q)^{\frac{1-s_c}{s_c}}$ and it is a non-scattering global solution.
\end{comment}
\item The condition on the sign of $K_\gamma(u_0)$ may be replaced by the sign of $\sqrt{2}\|Q_{1,0}\|_{L^2}^{1-s_c}\|Q_{1,0}\|_{\dot{H}^1}^{s_c}-\|u_0\|_{L^2}^{1-s_c}\|u_0\|_{\dot{H}_\gamma^1}^{s_c}$ in Theorem \ref{thm1.5} (in the low frequency case $\omega \leq \gamma^2/4$) and Theorem \ref{thm}. The latter condition is related to the Nehari functional as seen in Section \ref{sec2.4}.  See Appendix \ref{secB} for the proof. 
\item
In~\cite{GuIn23p}, even, logarithmic two-solitons with action $2n_{\omega,0}$ are constructed for higher frequencies.
Theorem~\ref{thm} shows that no such
solutions can exist at lower frequency.
\end{enumerate}
\end{remark}

We give a summary of global dynamics results for~\eqref{deltaNLS} in Section~\ref{secA}. 

\subsection{Idea of proof}

The basic idea to show the scattering result relies on the work by Duyckaerts, Landoulsi, and Roudenko \cite{DLR22} (see also \cite{MMZ21,ArIn22,GuIn22}). 
Their argument is by contradiction, based on concentration compactness and modulation. First of all, 
it is enough to consider the case $\omega=1$ by a scaling argument. 
We suppose there exists a threshold, even, non-scattering solution to \eqref{deltaNLS} satisfying the mass and energy condition $M(u)=2M(Q)$ and $E_\gamma(u)=2E(Q)$ and $K_\gamma(u)>0$. 
This solution has a compactness property:
there exists $x:[0,\infty) \to [0,\infty)$ such that for any $\varepsilon>0$ there exists $R=R_\varepsilon>0$ such that 
\begin{align*}
	\int_{\{|x-x(t)|>R\} \cap \{|x+x(t)|>R\}} |u'(t)|^2 +|u(t)|^2 dx <\varepsilon
\end{align*}
for any $t>0$. If $\mu_\gamma(u):=2\|Q'\|_{L^2}^2-\|u\|_{\dot{H}_\gamma^1}^2 > c$ for some $c>0$, then we get a contradiction in a similar way to the case below the ground state. Indeed, when we calculate the localized virial identity, the error term can be estimated using the compactness, and can be absorbed by the virial functional $K_\gamma$, which does not go to zero by $\mu_\gamma(u)>c$.
On the other hand, if $\mu_\gamma(u) \to 0$,
we need a more careful modulation argument. In particular, we need to exploit the effect of the repulsive potential. Indeed, if  $\gamma=0$, we have a non-scattering solution converging to the ground state (see \cite{CFR22}). The repulsive effect appears through an estimate like
\begin{equation} \label{repulse}
	0 \leq \left(1-\frac{2}{|\gamma|}\right)e^{-2y(t)} \lesssim |\mu_\gamma(u)|,
\end{equation}
where $y(t)$ is the translation parameter of the modulation (for $\gamma < -2$; see below for $\gamma=-2$). This shows that if $\mu_\gamma(u(t))\to 0$, then $y(t) \to \infty$. On the other hand, we also have 
\begin{align*}
	|\dot{y}(t)| \lesssim |\mu_\gamma(u)|
\end{align*}
by the modulation analysis
(this estimate holds in general for NLS and is not influenced by the repulsive potential), implying that $y$ is bounded, giving a contradiction. 

One difficulty of our problem is that we cannot remove the translation parameter by even symmetry. In the higher dimensional case, radial symmetry removes the translation parameter since the radial Sobolev embedding $H_\rad^1(\mathbb{R}^d) \subset L^q(\mathbb{R}^d)$, where $2<q<2d/(d-2)$ and $d\geq 3$, is compact. In one dimension,
however, the embedding is not compact.

A second difficulty is that the scaling of $\dot{H}^1$ and of the delta interaction are different. Indeed, letting $f_\lambda(x)=\lambda^\alpha f(\lambda^\beta x)$, we have 
\begin{align*}
	\|f_\lambda'\|_{L^2}^2 = \lambda^{2\alpha+\beta}\|f'\|_{L^2}^2,
	~|f_\lambda(0)|^2 = \lambda^{2\alpha} |f(0)|^2.
\end{align*}
This means that the virial functional $K_\gamma$ differs from the  functional $\mu_\gamma$, which corresponds to the Nehari functional under the mass-energy condition
(while they are equivalent for NLS without potential).
However, we can show that there exists $c>0$ such that $|\mu_\gamma(u)| \leq c |K_\gamma(u)|$ under the assumption, by using a variational argument based on more general functionals than the Nehari and virial ones. 

A third difficulty is to obtain the repulsive effect from the delta interaction. In the odd case (see~\cite{GuIn22}), we use the odd function $Q(x-y)-Q(-x-y)$ as the imaginary ground state and we could show an estimate like $e^{-2y}\lesssim |\mu(u)|$. In our problem, the key point to use $Q(|x|-y)$. While it is not smooth, it is more natural than $Q(x-y)+Q(-x-y)$, since the ground state of \eqref{deltaNLS} is given by  $Q(|x|-\xi)$ for some $\xi$ when $\omega>\gamma^2/4$ (see Remark \ref{rmk1.2}). 
In this way, we obtain an estimate like~\eqref{repulse}. The $e^{-2y}$ here comes from the decay order $e^{-x}$ of the ground state $Q(x)$. When $\gamma=-2$, 
however, the coefficient of $e^{-2y}$ vanishes. In this degenerate case, we need a more careful calculation, using the next order in the decay of $Q$, which is $e^{-px}$. In this way, we get
\begin{align*}
	e^{-(p+1)y} \lesssim |\mu_\gamma(u)|,
\end{align*}
which allows for a contradiction in the same way as for $\gamma<-2$. 

The proof of the blow-up result is based on 
the contradiction argument of Duyckaerts and Roudenko \cite{DuRo10}. Suppose that $u$ is global in the positive time direction. Then, by the virial identity and a Cauchy--Schwarz type inequality, which can be shown by using a Gagliardo--Nirenberg type inequality related to the delta potential, we know that $u$ blows up in finite negative time and satisfies $\int_{t}^{\infty} |\mu_\gamma(u(s))|ds \leq Ce^{-ct}$. The estimate implies that $y(t)$ converges to a positive constant as $t \to \infty$. This contradicts the divergence of $y$, which follows from $e^{-2y} \lesssim |\mu_\gamma(u)|$ (or $e^{-(p+1)y} \lesssim |\mu_\gamma(u)|$).

%They combined concentration compactness method and modulation argument. 
%We also apply their arguments to \eqref{deltaNLS}. 
%
%By the modulation argument, we have translation modulation parameter $y(t)$. On the other hand, we also have translation parameter $x(t)$, which is an element to break the compactness of the Sobolev embedding. To get the parameter $x(t)$, we use the concentration compactness. We note that the symmetry, evenness, is not sufficient to remove the parameter $x(t)$ since we consider the one dimensional case. Combining $x$ and $y$, we get the parameter $X(t)$, which still satisfies the compactness property of the solution. As in \cite{DLR22}, we will show $X(t)$ is unbounded and bounded, which gives us a contradiction. 
%
%ここらへんの証明のアイデアと難しさを書けるように整理すること。
%
%Our key point is to consider the modulation around $Q(|\cdot|-y)$ for $y>0$. Since $Q(|\cdot|-y)$ satisfies the elliptic equation with the delta potential depending on $y$, the function is more natural than, for example, a smooth function $Q(\cdot-y)+Q(x+y)$. Indeed, the ground state when $\omega>\gamma^2/4$ is given by the $Q(|\cdot|-y)$ for some $y$. 
%
%
%書きたいこと：
%
%1次元特有の困難さ。
%
%$\dot{H}^1$とデルタでスケールが違う難しさがある。これからKと$\mu$の違いを調べる必要が出てくる。
%
%$Q(|\cdot|-y)$を使うところ
%
%$\gamma=-2$がdegenerate caseであること. 

%%%%%%%%%%%%%%%%%%%%%%%%%%%%%%%%%%%%%%%%%%
\subsection{Notation}
For a function $f$, we denote the spatial derivative by $f'$ and the time derivative by $\dot{f}$. 

We set norms
\begin{align*}
	\|f\|_{H_{\omega,\gamma}^1}^2:=\|f'\|_{L^2}^2 +|\gamma||f(0)|^2+\omega \|f\|_{L^2}^2
\end{align*}
and 
\begin{align*}
	\|f\|_{\dot{H}_{\gamma}^1}^2:= \|\sqrt{-\Delta_\gamma}f\|_{L^2}^2= \|f'\|_{L^2}^2 +|\gamma||f(0)|^2.
\end{align*}
In what follows, we denote $Q_{\omega,0}$ by $Q_{\omega}$ and $Q_{1,0}$ by $Q$ for simplicity. 
%We define
%\begin{align*}
%	\mathcal{R}_yQ(x):=Q(x-y)+Q(-x-y)=Q(x-y)+Q(x+y)
%\end{align*}
%for $y>0$. 
Let $\chi \in C^\infty(\mathbb{R})$ be a cut-off function satisfying
\begin{align*}
	\chi(x) := 
	\begin{cases}
	1 & (|x|>1),
	\\
	0 & (|x|<1/2).
	\end{cases}
\end{align*}
For $R>0$, we set
\begin{align*}
	\chi_R(x):=\chi\left( \frac{x}{R}\right).
\end{align*}
We also use the following notations: 
\begin{align*}
	\chi_R^c(x):=1-\chi_R(x),~ 
	\chi_R^+ (x) := \mathbbm{1}_{(0,\infty)} \chi_R, \text{ and }
	\chi_R^- (x) := \mathbbm{1}_{(-\infty,0)} \chi_R.
\end{align*}
We define 
\begin{align*}
	\mu_\gamma(f)= 2\|Q'\|_{L^2}^2 - \|f\|_{\dot{H}_\gamma^1}^2. 
\end{align*}
For a function $f$ and $y>0$, we set $\mathcal{T}_y f(x):=f(x-y)$ and 
\begin{align*}
	\mathcal{G}_{R,y}f(x):=\chi_{R}^{+}(x)\mathcal{T}_{y}f(x) + \chi_{R}^{-}(x)\mathcal{T}_{-y}f(x).
\end{align*}.

%%%%%%%%%%%%%%%%%%%%%%%%%%%%%%%%%%%%%%%%%%
%%%%%%%%%%%%%%%%%%%%%%%%%%%%%%%%%%%%%%%%%%

\section{Preliminaries}

%%%%%%%%%%%%%%%%%%%%%%%%%%%%%%%%%%%%%%%%%%

\subsection{Variational structure}

We revisit the variational argument not only for the 
virial and Nehari functionals but also for more general functionals $K_{\omega,\gamma}^{\alpha,\beta}$, defined by
\begin{align*}
	K_{\omega,\gamma}^{\alpha,\beta}(f)
	&:=\partial_{\lambda}S_{\omega,\gamma}(e^{\alpha \lambda}f(e^{\beta\lambda}\cdot))|_{\lambda=0}
	\\
	&=\frac{2\alpha+\beta}{2}\|\partial_{x}f\|_{L^{2}}^{2}
	-\alpha\gamma|f(0)|^{2}
	+\omega \frac{2\alpha-\beta}{2}\|f\|_{L^{2}}^{2}
	-\frac{(p+1)\alpha-\beta}{p+1}\|f\|_{L^{p+1}}^{p+1}
\end{align*}
for $\alpha,\beta \in \mathbb{R}$. 
By definition, $(\alpha,\beta)=(1/2,1)$ gives the virial functional, i.e., $K_{\gamma}(f)=K_{\omega,\gamma}^{1/2,1}(f)$, and $(\alpha,\beta)=(1,0)$ gives the Nehari functional
\begin{align*}
	I_{\omega,\gamma}(f):=K_{\omega,\gamma}^{1,0}(f)
	=\|\partial_{x}f\|_{L^{2}}^{2} -\gamma|f(0)|^{2} +\omega\|f\|_{L^{2}}^{2} -\|f\|_{L^{p+1}}^{p+1}. 
\end{align*}

If  $(\alpha,\beta)$ satisfies
\begin{align}
\label{eq2.1}
	\alpha>0,~2\alpha-\beta\geq 0,~2\alpha+\beta\geq0,
\end{align}
we obtain the following lemma by \cite[Lemmas 2.7--2.10]{IkIn17}. 

\begin{lemma}
\label{lem2.3}

Let $0 < \omega\leq \gamma^2/4$ and $(\alpha,\beta)$ satisfy \eqref{eq2.1}. 
We have
\begin{align*}
	2S_{\omega,0}(Q_{\omega,0})
	=\inf \{S_{\omega,\gamma}(f):f\in H_{\even}^{1}(\mathbb{R})\setminus \{0\}, K_{\omega,\gamma}^{\alpha,\beta}(f)=0\}.
\end{align*}
\end{lemma}

For $0 < \omega\leq \gamma^2/4$ and $(\alpha,\beta)$ satisfying \eqref{eq2.1}, we define 
\begin{align*}
	\mathcal{K}_{\omega,\gamma}^{\alpha,\beta,+}
	&:=\{f\in H_{\even}^{1}(\mathbb{R}): M(f)=2M(Q_{\omega,0}), E_{\gamma}(f)=2E(Q_{\omega,0}), K_{\omega,\gamma}^{\alpha,\beta}(f)>0\},
	\\
	\mathcal{K}_{\omega,\gamma}^{\alpha,\beta,-}
	&:=\{f\in H_{\even}^{1}(\mathbb{R}): M(f)=2M(Q_{\omega,0}), E_{\gamma}(f)=2E(Q_{\omega,0}), K_{\omega,\gamma}^{\alpha,\beta}(f)<0\}.
\end{align*}

We note that there is no function satisfying $M(f)=2M(Q_{\omega,0}), E_{\gamma}(f)=2E(Q_{\omega,0})$, and $ K_{\omega,\gamma}^{\alpha,\beta}(f)=0$. 

\begin{lemma}
\label{lem2.2.0}
$\mathcal{K}_{\omega,\gamma}^{\alpha,\beta,+}$ is independent of $(\alpha,\beta)$; that is, $\mathcal{K}_{\omega,\gamma}^{\alpha,\beta,+}=\mathcal{K}_{\omega,\gamma}^{\alpha',\beta',+}$ for $(\alpha,\beta)$ and $(\alpha',\beta')$ satisfying \eqref{eq2.1}. A similar statement holds for $\mathcal{K}_{\omega,\gamma}^{\alpha,\beta,-}$. 
\end{lemma}

\begin{proof}
This can be proved in a similar way to \cite[Lemma 2.2]{GuIn22p}. We omit the proof. 
\end{proof}

This means that we may omit $(\alpha,\beta)$ from $\mathcal{K}_{\omega,\gamma}^{\alpha,\beta,\pm}$. This variational structure gives us  the invariant sets under the flow.

\begin{lemma}
\label{lem2.5}
Let $0 < \omega \leq \gamma^2/4$. 
Assume that $u_{0} \in H_\even^1(\mathbb{R})$ satisfies $E_\gamma(u_0)=2E(Q_{\omega,0})$ and $M(u_0)=2M(Q_{\omega,0})$. 
Let $u$ denote the solution to \eqref{deltaNLS} with $u(0)=u_{0}$.
\begin{enumerate}
\item If $K_{\gamma}(u_{0})>0$, then $u(t)\in \mathcal{K}_{\omega,\gamma}^{+}$ as long as the solution exists. 
This gives a uniform $H^1$ bound of the solution and thus the solution is global. 
\item If $K_{\gamma}(u_{0})<0$, then $u(t)\in \mathcal{K}_{\omega,\gamma}^{-}$ as long as the solution exists. 
\end{enumerate}
\end{lemma}

\begin{proof}
This follows from the standard argument. See e.g. \cite[Lemma 2.17]{IkIn17}. 
\end{proof}
%%%%%%%%%%%%%%%%%%%%%%%%%%%%%%%%%%%%%%%%%%

\subsection{The Gagliardo--Nirenberg inequality}

By the characterization of the ground state $Q_{\omega,0}$, we have the following Gagliardo--Nirenberg type inequality: 
\begin{align*}
	\|f\|_{L^{p+1}}^{2} \leq C_{\omega,0} \|f\|_{H_{\omega,0}^1}^2
\end{align*}
for $H^1(\mathbb{R})$, where 
\begin{align*}
	C_{\omega,0}^{-1}=\frac{\|Q_{\omega,0}\|_{H_{\omega,0}^1}^2}{\|Q_{\omega,0}\|_{L^{p+1}}^2}
	=\left\{ \frac{2(p+1)}{p-1}S_{\omega,0}(Q_{\omega,0})\right\}^{\frac{p-1}{p+1}}
	=\|Q_{\omega,0}\|_{L^{p+1}}^{p-1}
\end{align*}
is the best constant and it is attained by $Q_{\omega,0}$. 
For the repulsive Dirac delta potential, we have the following inequality. 

\begin{lemma}[The Gagliardo--Nirenberg type inequality w.r.t. delta potential]
\label{lem2.4}
Let $0<\omega \leq \gamma^2/4$. 
For any $f \in H_{\even}^1(\mathbb{R})$, the following estimate holds:
\begin{align*}
	\|f\|_{L^{p+1}}^{2} \leq 2^{-\frac{p-1}{p+1}}C_{\omega,0} \|f\|_{H_{\omega,\gamma}^1}^2.
\end{align*}
The constant is optimal, but it is not attained. 
\end{lemma}

\begin{proof}
This follows from the minimizing problem:
\begin{align*}
	2S_{\omega,0}(Q_{\omega,0})
	=\inf\{ S_{\omega,\gamma}(f): f\in H_\even^{1}(\mathbb{R})\setminus\{0\}, I_{\omega,\gamma}(f)=0\}.
\end{align*}
We omit the proof. 

\end{proof}

\begin{remark}
When $\omega > \gamma^2/4$, we have a different best constant. See \cite{GuIn23}. 
\end{remark}

%%%%%%%%%%%%%%%%%%%%%%%%%%%%%%%%%%%%%%%%%%

\subsection{Reduction by scaling}

By scaling, we can reduce our main theorem to the following for $\omega=1$: 
\begin{theorem}
\label{thm2.5}
Let $\gamma\leq -2$. 
Assume that $u_0 \in H_\even^1(\mathbb{R})$ satisfies the mass-energy condition
\begin{align}
\label{ME}
\tag{ME}
	M(u_0)=2M(Q) \text{ and } E_{\gamma}(u_0) =2E(Q).
\end{align}
Then we have
\begin{enumerate}[(1)]
\item If $K_\gamma(u_0)>0$, then the solution scatters in both time directions. 
\item Assume in additiona that $xu_0\in L^2(\mathbb{R})$. If $K_\gamma(u_0)<0$, then the solution blows up in finite time.
\end{enumerate}
\end{theorem}

Indeed, Theorem \ref{thm} can be shown by assuming Theorem \ref{thm2.5}. 

\begin{proof}[Proof of Theorem \ref{thm} from Theorem \ref{thm2.5}]
Let $u_0 \in H_{\even}^{1}(\mathbb{R})$ satisfy 
$E_\gamma(u_0)=2E(Q_{\omega,0})$, $M(u_0)=2M(Q_{\omega,0})$, and $K_\gamma(u_0)> 0$. 
By the scaling structure of $Q_{\omega,0}$, we obtain
\begin{align*}
	\omega^{-\frac{p+3}{2(p-1)}}E_\gamma(u_0)=2E(Q)
	\text{ and }
	\omega^{\frac{p-5}{2(p-1)}}M(u_0)=2M(Q).
\end{align*}
Let $u_{0,\omega^{-1}}(x):=\omega^{-1/(p-1)}u_{0}(\omega^{-1/2}x)$. Then we get
\begin{align*}
	E_{\omega^{-\frac{1}{2}}\gamma}(u_{0,\omega^{-1}})=2E(Q)
	\text{ and }
	M(u_{0,\omega^{-1}})=2M(Q).
\end{align*}
We also have
\begin{align*}
	K_{\omega^{-\frac{1}{2}}\gamma}(u_{0,\omega^{-1}})=\omega^{\frac{p+3}{2(p-1)}}K_\gamma(u_0)> 0
\end{align*}
By Theorem \ref{thm2.5}, if $\omega^{-1/2}\gamma\leq -2$, which is equivalent to $0<\omega \leq \gamma^2/4$, then we find that the solution $u_{\omega^{-1}}$ with the initial data $u_{0,\omega^{-1}}$ scatters. 
By rescaling, the behavior of the solution $u$ with $u(0)=u_0$  is same as that of  $u_{\omega^{-1}}$. Thus $u$ scatters. This argument also works for the case that $K_\gamma$ is  negative. 
\end{proof}

Thus, in what follows, it is enough to consider $\omega=1$ under the assumption $\gamma\leq -2$. 

\subsection{Relation between virial functional and $\mu_\gamma$}
\label{sec2.4}

We define 
\begin{align*}
	\mu_\gamma (f):=2\|Q'\|_{L^2}^2 - \|f\|_{\dot{H}_\gamma^1}^2. 
\end{align*}
This is nothing but the Nehari functional under the mass-energy condition \eqref{ME}, that is, $I_{1,\gamma}(u(t))= \frac{p-1}{2} \mu_\gamma(u(t))$. Thus this sign is invariant under the flow by Lemma \ref{lem2.5}.  Since the scaling ratio between $\dot{H}^1$-norm and $\delta$-interaction is different, there is a difference between the virial functional $K_\gamma$ and $\mu_\gamma$. However, we have Proposition \ref{prop2.7} below. The proof is similar to that in \cite{GuIn22p}. We set $K_{\gamma}^{\alpha,\beta}:=K_{1,\gamma}^{\alpha,\beta}$

\begin{lemma}
\label{lem2.6}
Let $u$ be the solution to \eqref{deltaNLS} with $u(0)=u_{0}$ satisfying \eqref{ME}. Then
\begin{align*}
	K_{\gamma}^{\alpha,\beta}(u(t)) 
	&= \frac{(p-1)\alpha-2\beta}{2} 2\|Q'\|_{L^{2}}^{2} 
	\\
	&\quad -\left(  \frac{(p-1)\alpha-2\beta}{2}\|\partial_x u(t)\|_{L^{2}}^{2} +\frac{(p-1)\alpha-\beta}{2}|\gamma||u(t,0)|^{2}\right)
\end{align*}
for any $t$. 
\end{lemma}

\begin{proof}
This follows from a direct calculation using \eqref{ME} and $K_0^{\alpha,\beta}(Q)=0$.
\end{proof}

\begin{lemma}
\label{lem2.8}
We have $K_{\gamma}(f) - c \mu_\gamma(f)=K_{\gamma}^{\frac{1}{2}-\frac{2c}{p-1},1}(f)$ for $c\in \mathbb{R}$ and $f \in H^1$ satisfying \eqref{ME} . 
\end{lemma}

\begin{proof}
This follows from direct calculation and Lemma \ref{lem2.6}.
\end{proof}

We note that $(\frac{1}{2}-\frac{2c}{p-1},1)$ does not satisfy \eqref{eq2.1}. However, the functional $K_{\gamma}^{\frac{1}{2}-\frac{2c}{p-1},1}$ can be used as the functional of the minimizing problem by taking $c$ small depending on $p$.

\begin{lemma}
\label{lem2.9}
Let $0<c<(p-5)/4$. 
There is no function $f \in H_\even^1$ satisfying \eqref{ME} and $K_{\gamma}^{\frac{1}{2}-\frac{2c}{p-1},1}(f)=0$.
\end{lemma}

\begin{proof}
The proof is same as in \cite[Lemma 2.7]{GuIn22p}. We note that there is no ground state if $\gamma\leq -2$. 
\end{proof}

\begin{proposition}
\label{prop2.7}
We have the following statements.
\begin{enumerate}
\item 
If $K_{\gamma}(u_{0})>0$, then there exists $c=c(p,u_{0})>0$ such that $K_{\gamma}(u(t))\geq  c \mu(u(t))>0$ for $t \in \mathbb{R}$. 
\item If $K_{\gamma}(u_{0})<0$, then there exists $c=c(p,u_{0})>0$ such that $K_{\gamma}(u(t)) \leq c \mu(u(t))<0$ as far as the solution exists.
\end{enumerate}
\end{proposition}

\begin{proof}[Proof of Proposition \ref{prop2.7}]
The proof is same as in \cite[Proposition 2.5]{GuIn22p}. 
%\begin{comment}
%Let $K_{\gamma}(u_{0})>0$. 
%By Lemma \ref{lem2.8}, we have
%\begin{align*}
%	K_{\gamma}(u(t)) -c \mu (u(t)) = K_{\omega,\gamma}^{\frac{1}{2}-\frac{2c}{p-1},1}(u(t)).
%\end{align*}
%Suppose that there exists a time $t_{*}$ such that $K_{\omega,\gamma}^{\frac{1}{2}-\frac{2c}{p-1},1}(u(t_{*}))=0$. Then we find that $u(t_{*})=e^{i\theta}Q_{\omega,\gamma}$ by Lemma \ref{lem2.9}. It follows that $K_{\gamma}(u(t_{*}))=0$, contradicting $K_{\gamma}(u_{0})>0$. Therefore, $K_{\omega,\gamma}^{\frac{1}{2}-\frac{2c}{p-1},1}(u(t))\neq 0$  for the time of existence. Now, we have
%\begin{align*}
%	K_{\omega,\gamma}^{\frac{1}{2}-\frac{2c}{p-1},1}(u_{0}) \to K_{\gamma}(u_{0})>0 \text{ as } c \to 0.
%\end{align*}
%Hence $K_{\omega,\gamma}^{\frac{1}{2}-\frac{2c}{p-1},1}(u_{0})>0$ if $c$ is sufficiently small depending on $u_{0}$. Therefore, we obtain $K_{\omega,\gamma}^{\frac{1}{2}-\frac{2c}{p-1},1}(u(t))>0$ for the time of existence, if $c=c(p,u_{0})>0$ is sufficiently small. So there is $c=c(p,u_{0})>0$ such that $K_{\gamma}(u(t))\geq  c \mu(u(t))>0$ for $t \in \mathbb{R}$. 
%The case $K_{\gamma}(u_{0})<0$ is proved in the same way.
%\end{comment}
\end{proof}

%%%%%%%%%%%%%%%%%%%%%%%%%%%%%%%%%%%%%%%%%%

\subsection{Virial identity}

Let $u$ be a solution to \eqref{deltaNLS}. We define
\begin{align*}
	J(u(t))=J_\infty(u(t)):=\int_{\mathbb{R}} |x|^2|u(t,x)|^2dx.
\end{align*}
Direct calculations give the virial identity: 
\begin{align*}
	\frac{d}{dt} J(u(t))&= 4 \im \int_{\mathbb{R}} xu'(t,x)\overline{u(t,x)}dx,
	\\
	\frac{d^2}{dt^2} J(u(t))&= 8K_\gamma(u(t)).
\end{align*}
Let $\varphi$ be an even function in $C_0^{\infty}(\mathbb{R})$ satisfying
\begin{align*}
	\varphi(x):= 
	\begin{cases}
	x^2 &(|x|<1),
	\\
	0 &(|x|>2).
	\end{cases}
\end{align*}
For $R>0$, we define 
\begin{align*}
	J_R(u(t)):= \int_{\mathbb{R}} R^2\varphi\left(\frac{x}{R}\right) |u(t,x)|^2dx. 
\end{align*}
Then we have the localized virial identity:
\begin{align*}
	\frac{d}{dt}J_R(u(t))= 2R\im \int_{\mathbb{R}} \varphi'\left(\frac{x}{R}\right) u'(t,x) \overline{u(t,x)}dx
\end{align*}
and 
\begin{align*}
	\frac{d^2}{dt^2}J_R(u(t)):= 8K_\gamma(u(t)) +A_R(u(t)),
\end{align*}
where
\begin{align*}
	A_R(u(t))&:=-4\int_{|x|>R}\left\{ 2-\varphi''\left(\frac{x}{R}\right) \right\}
	\left\{|u'(t,x)|^2 - \frac{p-1}{2(p+1)}|u(t,x)|^{p+1} \right\} dx
	\\
	&\quad -\frac{1}{R^2}\int_{R<|x|<2R} \varphi^{(4)}\left(\frac{x}{R}\right)|u(t,x)|^2dx.
\end{align*}
We set 
\begin{align*}
	F_R(f):= 8K_\gamma(f) +A_R(f)
\end{align*}
for $f \in H^1(\mathbb{R})$. We set $F_\infty(f) := 8K_\gamma(f)$. 

%\begin{lemma}
%Let $\theta\in \mathbb{R}$ and $y>0$. 
%We have 
%\begin{align*}
%	F_\infty(e^{i\theta} Q(|\cdot|-y))=
%	K_\gamma(e^{i\theta} Q(|\cdot|-y))
%	=\left( \frac{|\gamma|}{2}-1 \right)c_p^2 e^{-2y}
%	+ \frac{2(p+1-|\gamma|)}{p-1}c_p^2 e^{-(p+1)y} +o(e^{-(p+1)y}).
%\end{align*}
%\end{lemma}
%
%\begin{proof}
%We have
%\begin{align*}
%	K_\gamma(e^{i\theta} Q(|\cdot|-y))
%	= K_0(e^i\theta Q(|\cdot|-y)) + \frac{|\gamma|}{2}Q(-y)^2.
%\end{align*}
%Here, since $K_0(Q)=0$, we have
%\begin{align*}
%	K_0(e^{i\theta} Q(|\cdot|-y))
%	&= 2\left( \int_{0}^{\infty} |Q'(x-y)|^2 dx - \frac{p-1}{2(p+1)} \int_{0}^{\infty} |Q(x-y)|^{p+1}dx\right) \\
%	&=2K_0(Q)
%	-2 \left( \int_{-\infty}^{0} |Q'(x-y)|^2 dx -\frac{p-1}{2(p+1)} \int_{-\infty}^{0} |Q(x-y)|^{p+1}dx\right) \\
%	&=-2 \left( \int_{-\infty}^{0} |Q'(x-y)|^2 dx -\frac{p-1}{2(p+1)} \int_{-\infty}^{0} |Q(x-y)|^{p+1}dx\right).
%\end{align*}
%By Lemmas \ref{}, we get
%\begin{align*}
%	 &\frac{|\gamma|}{2}Q(-y)^2
%	 -2 \left( \int_{-\infty}^{0} |Q'(x-y)|^2 dx -\frac{p-1}{2(p+1)} \int_{-\infty}^{0} |Q(x-y)|^{p+1}dx\right) \\
%	& =\left( \frac{|\gamma|}{2}-1 \right)c_p^2 e^{-2y}
%	+ \frac{2(p+1-|\gamma|)}{p-1}c_p^2 e^{-(p+1)y} +o(e^{-(p+1)y}).
%\end{align*}
%This completes the proof. 
%\end{proof}

\begin{lemma}
\label{lem2.10}
For any $y > 0$ and $\theta \in \mathbb{R}$, we have
\begin{align*}
	A_R(e^{i\theta}Q(|\cdot|-y))=0
\end{align*}
and so
\begin{align*}
	F_R(e^{i\theta}Q(|\cdot|-y))
	=8K_\gamma (e^{i\theta}Q(|\cdot|-y))=F_\infty(e^{i\theta}Q(|\cdot|-y)).
\end{align*}
\end{lemma}
\begin{proof}
Let $v(t):=e^{it} e^{i\theta}Q(|\cdot|-y)$. Then the function $v$ satisfies 
\begin{align*}
	i\partial_t v + \partial_x^2 v + \gamma_y \delta v +|v|^{p-1}v=0,
\end{align*}
where $\gamma_y:= -2Q'(-y)/Q(-y)$ (in fact $v$ is the ground state soliton of~\eqref{deltaNLS} with $\gamma = \gamma_y \in (-2,0)$). 
Thus
\begin{align*}
	\frac{d}{dt} J_R(v(t))=2R \im \int_{\mathbb{R}} \varphi'\left( \frac{x}{R}\right) Q'(|x|-y)Q(|x|-y) dx=0.
\end{align*}
This shows that 
\begin{align*}
	F_R^y(v):= A_R(v) +8K_{\gamma_y} (v)=0. 
\end{align*}
By the Pohozaev identity for $Q(|\cdot|-y)$, we see that $K_{\gamma_y} (v)=0$, which implies that $A_R(v)=0$ (noting that $A_R$ is independent of the delta potential),
and the lemma follows. 
% Since $K_{\gamma_y} (v)=0$ and $F_R^y(v)=0$, we get
% \begin{align*}
% 	F_R(v)
% 	&=A_R(v) +8K_{\gamma} (v) \\
% 	&=A_R(v) +8K_{\gamma_y} (v) +8K_{\gamma} (v)\\
% 	&=F_R^y(v)  +8K_{\gamma} (v)\\
% 	&=8K_{\gamma} (v)
% \end{align*}
% We obtain the statement. 
\end{proof}

% \begin{remark}
% In fact, $A_R(Q(|\cdot|-y))=0$ is already shown in the previous work \cite[Lemma 2.9]{GuIn22p}. 
% \end{remark}

%%%%%%%%%%%%%%%%%%%%%%%%%%%%%%%%%%%%%%%%%%

\subsection{Coercivity}
\label{sec2.6}

We define 
\begin{align*}
	\Phi(f,g):=\re \int_{\mathbb{R}} f'(x)\overline{g'(x)} + f(x)\overline{g(x)}
	-Q(x)^{p-1}(pf_1(x)g_1 (x)+f_2(x)g_2 (x) ) dx,
\end{align*}
where $f,g \in H^1(\mathbb{R})$, $f_1,g_1$ denote the real parts of $f,g$, and $f_2,g_2$ denote the imaginary parts of $f,g$, respectively. 

For $y >0$, we also define 
\begin{align*}
	B_y(f,g):= \re \int_{0}^{\infty} f'(x)\overline{g'(x)} + f(x)\overline{g(x)}
	-Q(x-y)^{p-1}(pf_1(x)g_1 (x)+f_2(x)g_2 (x) ) dx
\end{align*}
and we denote $B_y$ by $B$ for short. 

The following coercivity property for $\Phi$ is obtained by \cite{CFR22}. 

\begin{lemma}[Campos--Farah--Roudenko {\cite[Lemma 3.5]{CFR22}}]
\label{lem2.11}
There exists $c > 9$ such that 
if $f\in H^1(\mathbb{R})$ satisfies the following orthogonality:
\begin{align*}
	\im \int_{\mathbb{R}} f(x)Q(x) dx 
	=\re \int_{\mathbb{R}} f(x)Q'(x) dx 
	=\re \int_{\mathbb{R}} f(x)Q(x)^p dx 
	=0,
\end{align*}
then
\begin{align*}
	\Phi(f,f) \geq c \|f\|_{H^1}^2. 
\end{align*}
\end{lemma}

By using this coercivity, we get the following. 

\begin{lemma}[Coercivity]
\label{lem2.12}
Let $y>2R$ and $R>1$. 
There exist $c,C>0$ such that if
$h \in H_\even^1(\mathbb{R})$ satisfies the following orthogonality:
\begin{align*}
	\im \int_{\mathbb{R}} h(x)\chi_R^+(x)\mathcal{T}_yQ(x) dx 
	&=\re \int_{\mathbb{R}} h(x) (\chi_R^+(x)\mathcal{T}_yQ(x))' dx 
	\\
	&=\re \int_{\mathbb{R}} h(x)\chi_R^+(x)\mathcal{T}_yQ(x)^p dx 
	=0,
\end{align*}
then
\begin{align*}
	\Phi(\mathcal{T}_{-y}(\chi_R h),\mathcal{T}_{-y}(\chi_R h)) \geq  c\|\chi_Rh\|_{H^1}^2-\frac{C}{R}\|h\|_{H^1}^2. 
\end{align*}
\end{lemma}

\begin{proof}
We can show this in the same way as in \cite[Lemma 24]{GuIn23}. We omit the proof. 
\end{proof}

\subsection{Strichartz estimates and linear profile decomposition}

We define some function spaces as follows. 
\begin{definition}
Let $I$ be a (possibly unbounded) time interval. We define function spaces by
\begin{align*}
	S(I):=L_{t}^{\frac{2(p^{2}-1)}{p+3}} L_{x}^{p+1}(I),
	\quad
	W(I):=L_{t}^{\frac{2(p^{2}-1)}{p^{2}-3p-2}} L_{x}^{p+1}(I),
	\quad
	X(I):=L_{t}^{p-1} L_{x}^{\infty}(I).
\end{align*}
Moreover, $W'$ denotes the dual space of $W$, that is, $W'(I)=L_{t}^{\frac{2(p^{2}-1)}{p(p+3)}} L_{x}^{\frac{p+1}{p}}(I)$.
\end{definition}

Then we have the Strichartz estimates.

\begin{lemma}[Strichartz estimates]
The following estimates are valid: 
\begin{align*}
	&\|e^{it\Delta_{\gamma}}f\|_{S(I)}+\|e^{it\Delta_{\gamma}}f\|_{X(I)} \lesssim \|f\|_{H^{1}}
	\\
	&\|\int_{t_{0}}^{t} e^{i(t-s)\Delta_{\gamma}}F(s)ds\|_{S(t_{0},t_{1})}
	+\|\int_{t_{0}}^{t} e^{i(t-s)\Delta_{\gamma}}F(s)ds\|_{X(t_{0},t_{1})}
	\lesssim \|F\|_{W'(t_{0},t_{1})},
\end{align*}
where $I$ is a (possibly unbounded) time interval and $t_{1}>t_{0}$. 
\end{lemma}

\begin{proof}
See \cite[Section 3.1]{BaVi16}. 
\end{proof}

\begin{lemma}
Let $u_0 \in H^1(\mathbb{R})$ and $u$ be the solution to \eqref{deltaNLS} with $u(0)=u_0$. The solution $u$ scatters in the positive time direction iff $u \in S(0,\infty)$. 
\end{lemma}
 
\begin{proof}
See \cite[Proposition 3.2]{BaVi16}. 
\end{proof}

%%%%%%%%%%%%%%%%%%%%%%%%%%%%%%%%%%%%%%%%%%

\section{Modulation argument}
\label{sec3}

\begin{lemma}
There exists $\mu_{0}>0$ and a function $\varepsilon:(0,\mu_{0}) \to (0,\infty)$ with $\varepsilon(\mu) \to 0$ as $\mu \to 0$ such that the following holds. For any $\mu < \mu_0$ and for all $f\in H_{\even}^{1}(\mathbb{R})$ satisfying $E_\gamma(f)=2E(Q)$, $M(f)=2M(Q)$ and $\mu_\gamma(f)<\mu$,  there exist $(\theta, y)\in \mathbb{R} \times [0,\infty)$ such that
\begin{align*}
	\| f-e^{i\theta}Q(|\cdot|-y)\|_{H^{1}}\leq \varepsilon(\mu).
\end{align*} 
\end{lemma}

\begin{proof}
We use a contradiction argument. We suppose that the statement fails. 
Then there exists $\varepsilon_0>0$ such that for any $n\in \mathbb{N}$, there exist $\mu_n$ with $\mu_n \to 0$ and $f_n \in H_{\even}^{1}(\mathbb{R})$ satisfying $E_\gamma(f_n)=2E(Q)$, $M(f_n)=2M(Q)$ and $\mu(f_n)<\mu_n$ such that $ \inf_{\theta \in \mathbb{R}}\inf_{y\geq 0}\| f_n-e^{i\theta}Q(|\cdot|-y)\|_{H^{1}} > \varepsilon_0$. 
Since $\mu_\gamma(f_n) \to 0$ as $n \to \infty$ and $E_\gamma(f_n)=2E(Q)$, we have $\|f_n\|_{L^{p+1}}^{p+1} \to 2\|Q\|_{L^{p+1}}^{p+1}$ and thus
\begin{align*}
	S_\gamma(f_n)  \to 2S(Q) \text{ and } I_\gamma(f_n)\to 2I(Q)=0.
\end{align*}
By Fukuizumi and Jeanjean \cite[Lemma 20]{FuJe08}, we get
\begin{align*}
	\|f_n - e^{i\theta_n}(Q(\cdot-y_n)+Q(-\cdot-y_n)) \|_{H^{1}} \to 0,
\end{align*}
where $y_n \to \infty$. 
Now
\begin{align*}
	&\| (Q(\cdot-y_n)+Q(-\cdot-y_n)) - Q(|\cdot|-y_n) \|_{H^{1}}
	\\
	&=\| Q(-\cdot-y_n) \|_{H^{1}(0,\infty)}
	+\| Q(\cdot-y_n) \|_{H^{1}(-\infty,0)}
	\\
	& \to 0.
\end{align*}
This is a contradiction. 
\end{proof}

\begin{lemma}[Modulation]
\label{lem4.2}
Let $R>0$ be sufficiently large. 
There exist $\mu_0>0$ and a function $\varepsilon:(0,\mu_{0}) \to (0,\infty)$ with $\varepsilon(\mu) \to 0$ as $\mu \to 0$ such that the following holds. For any $\mu < \mu_0$ and for all $f\in H_{\even}^{1}(\mathbb{R})$ satisfying $E(f)=2E(Q)$, $M(f)=2M(Q)$ and $\mu(f)<\mu$,  there exist $(\tilde{\theta}, y)\in \mathbb{R} \times (2R,\infty)$ such that
\begin{align*}
	\| e^{-i\tilde{\theta}}f- Q(|\cdot|-y)\|_{H^1} < \varepsilon(\mu)
\end{align*}
and
\begin{align}
\label{eq4.2}
	\im \int_{\mathbb{R}} g(x) \chi_{R}^{+}(x)Q(x-y)dx=0,
	\quad \re \int_{\mathbb{R}} g(x) \partial_x(\chi_{R}^{+}(x) Q(x-y))dx=0,
\end{align}
where $g=e^{-i\tilde{\theta}}f-Q(|\cdot|-y)$. 
\end{lemma}

\begin{proof}
We define 
\begin{align*}
	J(\tilde{\theta},y,v)&=
	\begin{pmatrix}
	J_1(\tilde{\theta},y,v)
	\\
	J_2(\tilde{\theta},y,v)
	\end{pmatrix}
	\\
	&:=
	\begin{pmatrix}
	\im \int_{\mathbb{R}} (e^{-i\tilde{\theta}}v - Q(|\cdot|-y) ) \chi_{R}^{+}\mathcal{T}_{y}Qdx
	\\
	\re \int_{\mathbb{R}} (e^{-i\tilde{\theta}}v - Q(|\cdot|-y) ) \partial_x(\chi_{R}^{+} \mathcal{T}_{y}Q)dx
	\end{pmatrix}
\end{align*}
for $\tilde{\theta} \in \mathbb{R}$, $y> 2 R$, $v \in H_{\even}^{1}$. Then $J(0,y,Q(|\cdot|-y))=0$. We have
\begin{align*}
	\frac{\partial J_1}{\partial \tilde{\theta}}(0,y,Q(|\cdot|-y))
	&=- \int_{\mathbb{R}} Q(|x|-y)  \chi_{R}^{+}\mathcal{T}_{y}Qdx
	\\
	&=-\int_{\mathbb{R}} |\mathcal{T}_{y}Q |^2 dx+O(e^{-2y})
	\\
	&=-\|Q \|_{L^2}^2+O(e^{-4R})
\end{align*}
and 
\begin{align*}
	\frac{\partial J_2}{\partial y}(0,y,Q(|\cdot|-y))
	&
	=\int_{\mathbb{R}}  Q'(|x|-y) \partial_x(\chi_{R}^{+} \mathcal{T}_{y}Q)dx
	\\
	&=\|\partial_x Q \|_{L^2}^2 +O(R^{-1} + e^{-4R}).
\end{align*}
We also have $\frac{\partial J_1}{\partial y}(0,y,Q(|\cdot|-y))=\frac{\partial J_2}{\partial \tilde{\theta}}(0,y,Q(|\cdot|-y))=0$.
These imply that the Jacobian of $J$ near $(0,y,Q(|\cdot|-y))$ is invertible for large $R$. The statement follows from the implicit function theorem. 
\end{proof}

Let $u$ be an even solution satisfying \eqref{ME}. 
We denote $\mu_\gamma(u(t))$ by $\mu_\gamma(t)$ for short. We set $I_{\mu_0}:=\{ t \in I_{\max} :  |\mu_\gamma(t)| < \mu_0\}$, where $I_{\max}$ denotes the maximal existence time interval of the solution. 
By Lemma \ref{lem4.2}, we have $C^1$ functions $\tilde{\theta}=\tilde{\theta}(t)$ and $y=y(t)$ for $t \in I_{\mu_0}$. We set $\theta:=\tilde{\theta}-t$. We also have orthogonality conditions \eqref{eq4.2}.
We set
\begin{align}
	\label{eq4.3}
	u(t,x)
	&=e^{i\theta(t) + i t}(Q(|x|-y(t)) + g(t,x))
	\\ \notag
	&=e^{i\theta(t) + i t}(Q(|x|-y(t)) + \rho(t)\mathcal{G}_{R,y(t)}Q(x) + h(t,x)),
\end{align}
where 
\begin{align}
\label{eq4.4}
	\rho(t):=\frac{\re \int g \chi_{R}^{+} (\mathcal{T}_{y(t)}Q )^p dx}{\int (\chi_{R}^{+})^2(\mathcal{T}_{y(t)}Q )^{p+1}dx}.
\end{align}
Then it follows from \eqref{eq4.2}, \eqref{eq4.3} and \eqref{eq4.4} that
\begin{align}
\label{ortho}
	\im \int_{\mathbb{R}} h \chi_{R}^{+}\mathcal{T}_{y(t)}Qdx
	=\re \int_{\mathbb{R}} h (\chi_{R}^{+} \mathcal{T}_{y(t)}Q)'dx
	=\re \int_{\mathbb{R}} h \chi_{R}^{+} (\mathcal{T}_{y(t)}Q)^{p}dx
	=0.
\end{align}

%%%%%%%%%%%%%%%%%%%%%%%%%%%%%%%%%%%%%%%%%
\section{Modulation parameters}

\subsection{Properties of the ground state}
To give estimates for the modulation parameters $\rho,y$ and the remainder terms $g,h$, the function $\mathcal{Q}$ defined by
\begin{align*}
	\mathcal{Q}(y)&:=-\int_{-\infty}^{0} |Q'(x-y)|^2dx
	- \int_{-\infty}^{0} |Q(x-y)|^2dx  + \frac{|\gamma|}{2} Q(-y)^2 \\
	&\quad -\frac{1}{2|\gamma|} \{|\gamma|Q(-y)-2 Q'(-y)\}^2
\end{align*}
plays an important role. In this subsection, we compute its function large-$y$ asymptotics. 
The ground state $Q$ is given by
\begin{align}
	Q(x)= c_p \left\{2 \cosh \left(\frac{p-1}{2} x \right) \right\}^{-\frac{2}{p-1}}
	=c_p\{e^{\alpha x} + e^{-\alpha x}\}^{-\frac{1}{\alpha}},
\end{align}
where $c_p:=  \left( 2(p+1)\right)^{\frac{1}{p-1}}$ and $\alpha=(p-1)/2$. 
By Taylor expansion, we have
\begin{align}
\label{Q}
	Q(x)=c_p e^{-|x|}-\frac{c_p}{\alpha}e^{-p |x|}+o(e^{-p |x|}),
\end{align}
and
\begin{align}
\label{Q'}
	Q'(x)=-c_p\frac{x}{|x|} e^{-|x|}+ \frac{pc_p}{\alpha}\frac{x}{|x|}e^{-p |x|}+o(e^{-p |x|}).
\end{align}

\begin{comment}
Indeed, we show it as follows. For $x>0$, we have
\begin{align*}
	Q(x)
	&=c_p(e^{\alpha x} + e^{-\alpha x})^{-\frac{1}{\alpha}} \\
	&=c_p \{e^{\alpha x}(1+e^{-2\alpha x})  \}^{-\frac{1}{\alpha}} \\
	&=c_p e^{-x}(1+e^{-2\alpha x})^{-\frac{1}{\alpha}}.
\end{align*}
Letting $f(t)=(1+t)^{-\frac{1}{\alpha}}$, we have
\begin{align*}
	f'(t)=-\frac{1}{\alpha}  (1+t)^{-\frac{1}{\alpha}-1}.
\end{align*}
By the Taylor expansion, we have $f(t)=f(0)+f'(0)t+o(t)$, that is,
\begin{align*}
	(1+t)^{-\frac{1}{\alpha}}=1-\frac{1}{\alpha} t+o(t).
\end{align*}
Applying this as $t=e^{-2\alpha x}$, since $\alpha=(p-1)/2$ we get
\begin{align*}
	Q(x)&=c_p e^{-x}\left(1-\frac{1}{\alpha}e^{-2\alpha x}+o(e^{-2\alpha x})\right) \\
	&=c_p e^{-x}-\frac{c_p}{\alpha}e^{(-2\alpha-1) x}+o(e^{(-2\alpha-1) x})
	\\
	&=c_p e^{-x}-\frac{c_p}{\alpha}e^{-p x}+o(e^{-p x}).
\end{align*}
By the symmetry, we have
\begin{align*}
	Q(x)=c_p e^{-|x|}-\frac{c_p}{\alpha}e^{-p |x|}+o(e^{-p |x|})
\end{align*}
for $x\in \mathbb{R}$. By the derivative, we have
\begin{align*}
	Q'(x)=-c_p\frac{x}{|x|} e^{-|x|}+ \frac{pc_p}{\alpha}\frac{x}{|x|}e^{-p |x|}+o(e^{-p |x|}).
\end{align*}
\end{comment}

\begin{lemma}
\label{lem4.1}
We have
\begin{align*}
	 - \int_{-\infty}^{0} |Q'(x-y)|^2dx  
	 = - \frac{c_p^2}{2} e^{-2y} +\frac{2pc_p^2}{\alpha(p+1)}e^{-(p+1)y} + o(e^{-(p+1)y}). 
\end{align*}
\end{lemma}

\begin{proof}
%Since we have $Q'(x)=-c_p\frac{x}{|x|} e^{-|x|}+ \frac{pc_p}{\alpha}\frac{x}{|x|}e^{-p |x|}+o(e^{-p |x|})$ and $x-y<0$ for $x<0$, 
By the Taylor expansion \eqref{Q'} of $Q'$, we get
\begin{align*}
	 - \int_{-\infty}^{0} |Q'(x-y)|^2dx 
%	& = - \int_{-\infty}^{0} |-c_pe^{-|x-y|}+ \frac{pc_p}{\alpha}e^{-p |x-y|}|^2dx + o(e^{-(p+1)y})
%	\\
	& = - \int_{-\infty}^{0} |-c_pe^{x-y}+ \frac{pc_p}{\alpha}e^{p(x-y)}|^2dx + o(e^{-(p+1)y})
	\\
	& = - \int_{-\infty}^{0} (-c_pe^{x-y})^2 +2(-c_pe^{x-y})\frac{pc_p}{\alpha}e^{p(x-y)} dx + o(e^{-(p+1)y})
	\\
	& = -c_p^2 e^{-2y} \int_{-\infty}^{0} e^{2x}dx +2\frac{pc_p^2}{\alpha}e^{-(p+1)y} \int_{-\infty}^{0} e^{(p+1)x}dx + o(e^{-(p+1)y})
	\\
	& = - \frac{c_p^2}{2} e^{-2y} +\frac{2pc_p^2}{\alpha(p+1)}e^{-(p+1)y} + o(e^{-(p+1)y}).
\end{align*}
\end{proof}
\begin{lemma}
\label{lem4.2}
We have
\begin{align*}
	- \int_{-\infty}^{0} |Q(x-y)|^2 dx
	= -\frac{c_p^2}{2} e^{-2y}+\frac{2c_p^2}{\alpha(p+1)}e^{-(p+1)y}  + o(e^{-(p+1)y}).
\end{align*}
\end{lemma}
\begin{proof}
Using the Taylor expansion \eqref{Q} of $Q$, this is a similar calculation to Lemma \ref{lem4.1}.  
\begin{comment}
We have 
\begin{align*}
	- \int_{-\infty}^{0} |Q(x-y)|^2 dx
	&= - \int_{-\infty}^{0} |c_p e^{x-y}-\frac{c_p}{\alpha}e^{p (x-y)}|^2 dx + o(e^{-(p+1)y})
	\\
	&= - \int_{-\infty}^{0} c_p^2 e^{-2y}e^{2x}+\frac{2c_p^2}{\alpha}e^{-(p+1)y}e^{(p+1)x} dx + o(e^{-(p+1)y})
	\\
	&= -c_p^2 e^{-2y} \int_{-\infty}^{0} e^{2x} dx+\frac{2c_p^2}{\alpha}e^{-(p+1)y} \int_{-\infty}^{0}  e^{(p+1)x} dx + o(e^{-(p+1)y})
	\\
	&= -\frac{c_p^2}{2} e^{-2y}+\frac{2c_p^2}{\alpha(p+1)}e^{-(p+1)y}  + o(e^{-(p+1)y})
\end{align*}
\end{comment}

\end{proof}

\begin{lemma}
\label{lem4.3}
We have
\begin{align*}
	\frac{|\gamma|}{2} Q(-y)^2=\frac{|\gamma|}{2} c_p^2 e^{-2y}- \frac{|\gamma|c_p^2}{\alpha}e^{-(p+1)y} +o(e^{-(p+1)y}).
\end{align*}
\end{lemma}
\begin{proof}
This is an immediate consequence of \eqref{Q}.
% \begin{align*}
% 	\frac{|\gamma|}{2} Q(-y)^2
% 	&=\frac{|\gamma|}{2} (c_p e^{-y}-\frac{c_p}{\alpha}e^{-py})^2 +o(e^{-(p+1)y})
%	\\
%	&=\frac{|\gamma|}{2} c_p^2 e^{-2y}- \frac{|\gamma|}{2} \frac{2c_p^2}{\alpha}e^{-(p+1)y} +o(e^{-(p+1)y})
% 	\\
% 	&=\frac{|\gamma|}{2} c_p^2 e^{-2y}- \frac{|\gamma|c_p^2}{\alpha}e^{-(p+1)y} +o(e^{-(p+1)y})
% \end{align*}
% This is the desired conclusion.
\end{proof}

\begin{lemma}
\label{lem4.4}
We have
\begin{align*}
	& - \frac{1}{2|\gamma|} \{|\gamma|Q(-y)-2Q'(-y)\}^2
	 \\
	 &= - \frac{1}{2|\gamma|}(|\gamma|-2)^2c_p^2 e^{-2y}  - \frac{|\gamma|-2}{|\gamma|}\frac{2p-|\gamma|}{\alpha}c_p^2 e^{-(p+1)y} +o(e^{-(p+1)y}).
\end{align*}
\end{lemma}

\begin{proof}
By  \eqref{Q} and  \eqref{Q'}, 
we have 
\begin{align*}
	&\{|\gamma|Q(-y)-2Q'(-y)\}^2 
	\\
	&=\left\{ |\gamma|\left( c_p e^{-y}-\frac{c_p}{\alpha}e^{-py}\right) -2\left(c_pe^{-y}- \frac{pc_p}{\alpha}e^{-py}\right) \right\}^2+o(e^{-(p+1)y})
	\\
	&=\left\{ (|\gamma|-2)c_p e^{-y} +\frac{2p-|\gamma|}{\alpha}c_pe^{-py}\right\}^2+o(e^{-(p+1)y})
	\\
	&=(|\gamma|-2)^2c_p^2 e^{-2y} +2(|\gamma|-2)\frac{2p-|\gamma|}{\alpha}c_p^2 e^{-(p+1)y} +o(e^{-(p+1)y})
\end{align*}
and ther result follows.
\end{proof}
Combining all the above estimates, we obtain the following. 
\begin{corollary}
\label{cor4.5}
We have
\begin{align*}
	\mathcal{Q}(y)
	=\left( 1 -\frac{2}{|\gamma|}\right)c_p^2 e^{-2y} 
	- \frac{4p(|\gamma|-2)}{|\gamma|(p-1)} c_p^2 e^{-(p+1)y}
	+o( e^{-(p+1)y}). 
\end{align*}
\end{corollary}

\begin{remark}
So when $\gamma<-2$, the leading term of $\mathcal{Q}$ is $e^{-2y}$ with  positive coefficient. When $\gamma=-2$, we will simply regard $\mathcal{Q}$ as an error term with $\mathcal{Q}=o(e^{-(p+1)y})$. 
\end{remark}

For $\gamma=-2$, the following estimate provides our main term: 
% At the last of this subsection, we give the following estimate.
\begin{lemma}
\label{lem4.6}
It holds that
\begin{align*}
	 \frac{2}{p+1} \int_{-\infty}^{0} |Q(x-y)|^{p+1} dx
	  =  \frac{4c_p^2}{(p+1)}  e^{-(p+1)y}  +o(e^{-(p+1)y}).
\end{align*}
\end{lemma}

\begin{proof}
We have
\begin{align*}
	 \frac{2}{p+1} \int_{-\infty}^{0} |Q(x-y)|^{p+1} dx 
	 &=  \frac{2}{p+1} \int_{-\infty}^{0} |c_pe^{x-y}|^{p+1} dx +o(e^{-(p+1)y})
	 \\
	  &=  \frac{2c_p^{p+1}}{p+1} e^{-(p+1)y} \int_{-\infty}^{0} e^{(p+1)x} dx +o(e^{-(p+1)y})
	  \\
	  &=  \frac{2c_p^{p+1}}{(p+1)^2} e^{-(p+1)y}  +o(e^{-(p+1)y})
\end{align*}
Since $c_p=\{2(p+1)\}^{\frac{1}{p-1}}$ we have
\begin{align*}
	c_p^{p+1} =c_p^{2}c_p^{p-1} = 2(p+1)c_p^{2}.
\end{align*}
Thus, we get
\begin{align*}
	 \frac{2c_p^{p+1}}{(p+1)^2} 
	 =  \frac{4(p+1)c_p^2}{(p+1)^2} 
	 =  \frac{4c_p^2}{(p+1)}.
\end{align*}
\end{proof}
% Actually, the precise value of the integral is not important. The key point is that its order is $e^{-(p+1)y}$ and it is positive. We will use this fact to get an estimate when $\gamma=-2$. 

%%%%%%%%%%%%%%%%%%%%%%%%%%%%%%%%%%%%%%%%%

\subsection{Estimates of the modulation parameters}

We set 
\begin{align*}
	e_\gamma(y):=
	\begin{cases}
	e^{-2y} & (\gamma<-2),
	\\
	e^{-(p+1)y}& (\gamma=-2). 
	\end{cases}
\end{align*}
Let $u$ be an even solution satisfying \eqref{ME}, and  modulated as in~\eqref{eq4.3}. 
\begin{lemma}
\label{lem4.7}
It holds that 
\begin{align*}
	&|\rho| \lesssim \|g\|_{L^2},\\
	& \|g\|_{H^1} \lesssim |\rho| + \|h\|_{H^1},\\
	&\|h\|_{H^1} \lesssim |\rho| +\|g\|_{H^1}  \lesssim \|g\|_{H^1}.
\end{align*}
\end{lemma}

\begin{proof}
These estimates follow from the definition of $\rho$ and $g=\chi_R Q(|\cdot|-y)+h$. 
\end{proof}

By using the mass condition, we have the following:
\begin{lemma}
\label{lem4.8}
We have
\begin{align*}
	4 \re \int_{0}^{\infty}  Q(x-y)h(x) dx
	= 2 \int_{-\infty}^{0} |Q(x-y)|^2dx - 4\rho \re \int_{0}^{\infty} \chi_R(x) Q(x-y)^2 -\|g\|_{L^2}^2.
\end{align*}
\end{lemma}
\begin{proof}
This follows from direct calculation using $M(u)=2M(Q)$ and~\eqref{eq4.3}. 
\begin{comment}
We have
\begin{align*}
	0=M(u)-2M(Q)&=M(Q(|\cdot|-y) +g)-2M(Q)
	\\
	&=- 2 \int_{-\infty}^{0} |Q(x-y)|^2dx +2 \re \int_{\mathbb{R}} Q(|x|-y)g(x) dx +\|g\|_{L^2}^2 \\
	&=- 2 \int_{-\infty}^{0} |Q(x-y)|^2dx +4 \re \int_{0}^{\infty} Q(x-y)g(x) dx +\|g\|_{L^2}^2 \\
\end{align*}
By $g=\rho \chi_R Q(|\cdot|-y) + h$, it is valid that 
\begin{align*}
	\re \int_{0}^{\infty} Q(x-y)g(x) dx 
	&= \re \int_{0}^{\infty} Q(x-y)(\rho \chi_R(x) Q(x-y) + h(x)) dx \\
	&=\rho \re \int_{0}^{\infty} \chi_R(x) Q(x-y)^2 
	+ \re \int_{0}^{\infty}  Q(x-y)h(x) dx.
\end{align*}
Therefore, combining these equalities, we get the statement. 
\end{comment}
\end{proof}

Using the functional $\mu_\gamma$, we have the following: 
\begin{lemma}
\label{lem4.9}
We have
\begin{align*}
	|\rho| \lesssim |\mu_\gamma(u)|
	+ \|g\|_{H^1}^2
	+e_\gamma(y) 
	+o(e^{-(p+1)y})
	+ O(e^{-pR}\|h\|_{H^1}).
\end{align*}
\end{lemma}
\begin{proof}
We have
\begin{align*}
	\mu_\gamma(u)
%	&=2 \|Q'\|_{L^2}^2 - \|u\|_{\dot{H}_\gamma^1}^2
%	\\
	&=2 \|Q'\|_{L^2}^2 -\|Q(|\cdot|-y)+g\|_{\dot{H}_\gamma^1}^2
	\\
	&=2 \|Q'\|_{L^2}^2 - 2\int_{0}^{\infty} |Q'(x-y)|^2dx 
	- 4 \re \int_{0}^{\infty} Q'(x-y)g'(x) dx 
	- \int_{\mathbb{R}} |g'(x)|^2 dx
	\\
	&\quad - |\gamma| |Q(-y)+g(0)|^2
\end{align*}
Now, 
\begin{align*}
	2 \|Q'\|_{L^2}^2 - 2\int_{0}^{\infty} |Q'(x-y)|^2dx
	= 2\int_{-\infty}^{0} |Q'(x-y)|^2dx.
\end{align*}
By equation~\ref{ellip} for $Q$ and integration by parts, we get
\begin{align*}
	\re \int_{0}^{\infty} Q'(x-y)g'(x) dx
	&= \re [Q'(x-y)g(x)]_{x=0}^{x=\infty} - \re \int_{0}^{\infty} Q''(x-y)g(x) dx
	\\
	&= -Q'(-y)\re g(0) - \re \int_{0}^{\infty} (Q(x-y)-Q(x-y)^p)g(x) dx.
\end{align*}
Since $g=\rho\chi_RQ(\cdot-y)+h$ on $(0,\infty)$, we get
\begin{align*}
	&\re \int_{0}^{\infty} (Q(x-y)-Q(x-y)^p)g(x) dx \\
%	&=\re \int_{0}^{\infty} (Q(x-y)-Q(x-y)^p)(\rho\chi_R(x)Q(x-y)+h(x)) dx
%	\\
%	&=\rho\int_{0}^{\infty} \chi_R(x)Q(x-y)^2dx \\
%	&\quad +\re \int_{0}^{\infty} Q(x-y)h(x) dx \\
%	&\quad - \rho \int_{0}^{\infty} \chi_R(x)Q(x-y)^{p+1} dx \\
%	&\quad -\re \int_{0}^{\infty}Q(x-y)^ph(x) dx
%	\\
	&=\rho \left( \int_{0}^{\infty} \chi_R(x)Q(x-y)^2dx- \int_{0}^{\infty} \chi_R(x)Q(x-y)^{p+1} dx\right) \\
	&\quad +\re \int_{0}^{\infty} Q(x-y)h(x) dx -\re \int_{0}^{\infty} Q(x-y)^ph(x) dx
\end{align*}
By the above equations and Lemma \ref{lem4.8}, 
%Here, we recall that
%\begin{align*}
%	4 \re \int_{0}^{\infty}  Q(x-y)h(x) dx
%	= 2 \int_{-\infty}^{0} |Q(x-y)|^2dx - 4\rho \re \int_{0}^{\infty} \chi_R(x) Q(x-y)^2 -\|g\|_{L^2}^2.
%\end{align*}
we obtain 
\begin{align*}
	\mu_\gamma(u)
%	&= 2\int_{-\infty}^{0} |Q'(x-y)|^2dx
%	- 4 \left( -Q'(-y)\re g(0) - \re \int_{0}^{\infty} (Q(x-y)-Q(x-y)^p)g(x) dx \right) \\
%	&- \int_{\mathbb{R}} |g'(x)|^2 dx
%	\\
%	&\quad - |\gamma| |Q(-y)+g(0)|^2\\
%	&= 2\int_{-\infty}^{0} |Q'(x-y)|^2dx
%	+ 4 Q'(-y)\re g(0)
%	+4 \re \int_{0}^{\infty} (Q(x-y)-Q(x-y)^p)g(x) dx \\
%	&- \int_{\mathbb{R}} |g'(x)|^2 dx
%	\\
%	&\quad - |\gamma| |Q(-y)+g(0)|^2\\
%	&=2\int_{-\infty}^{0} |Q'(x-y)|^2dx+4 Q'(-y)\re g(0) \\
%	&\quad +4 \rho \left( \int_{0}^{\infty} \chi_R(x)Q(x-y)^2dx- \int_{0}^{\infty} \chi_R(x)Q(x-y)^{p+1} dx\right)\\
%	&\quad +4\re \int_{0}^{\infty} Q(x-y)h(x) dx - 4\re \int_{0}^{\infty} Q(x-y)^ph(x) dx\\
%	&- \int_{\mathbb{R}} |g'(x)|^2 dx
%	\\
%	&\quad - |\gamma| |Q(-y)+g(0)|^2\\
%	&=2\int_{-\infty}^{0} |Q'(x-y)|^2dx+4 Q'(-y)\re g(0) \\
%	&\quad +4 \rho \left( \int_{0}^{\infty} \chi_R(x)Q(x-y)^2dx- \int_{0}^{\infty} \chi_R(x)Q(x-y)^{p+1} dx\right)\\
%	&\quad +2 \int_{-\infty}^{0} |Q(x-y)|^2dx - 4\rho \re \int_{0}^{\infty} \chi_R(x) Q(x-y)^2 -\|g\|_{L^2}^2 - 4\re \int_{0}^{\infty} Q(x-y)^ph(x) dx\\
%	&- \int_{\mathbb{R}} |g'(x)|^2 dx
%	\\
%	&\quad - |\gamma| |Q(-y)+g(0)|^2\\
%	&=2\int_{-\infty}^{0} |Q'(x-y)|^2dx+4 Q'(-y)\re g(0) \\
%	&\quad -4 \rho \int_{0}^{\infty} \chi_R(x)Q(x-y)^{p+1} dx\\
%	&\quad +2 \int_{-\infty}^{0} |Q(x-y)|^2dx  -\|g\|_{L^2}^2 - 4\re \int_{0}^{\infty} Q(x-y)^ph(x) dx\\
%	&- \int_{\mathbb{R}} |g'(x)|^2 dx
%	\\
%	&\quad - |\gamma| |Q(-y)+g(0)|^2 \\
	&=2\int_{-\infty}^{0} |Q'(x-y)|^2dx+2 \int_{-\infty}^{0} |Q(x-y)|^2dx +4 Q'(-y)\re g(0) - |\gamma| |Q(-y)+g(0)|^2 \\
	&\quad -4 \rho \int_{0}^{\infty} \chi_R(x)Q(x-y)^{p+1} dx
	-\|g\|_{H^1}^2 - 4\re \int_{0}^{\infty} Q(x-y)^ph(x) dx
	\\
	&=-2\mathcal{Q}(y) -2\mathrm{Err} 
	-4 \rho \int_{0}^{\infty} \chi_R(x)Q(x-y)^{p+1} dx
	-\|g\|_{H^1}^2 - 4\re \int_{0}^{\infty} Q(x-y)^ph(x) dx,
\end{align*}
where we set
\begin{align*}
	\mathrm{Err}:=\frac{1}{2|\gamma|} \{|\gamma|Q(-y)-2 Q'(-y)\}^2
	+\{|\gamma|Q(-y)-2 Q'(-y)\}\re g(0) + \frac{|\gamma|}{2}|g(0)|^2.
\end{align*}
Now, we have
\begin{align*}
	\re \int_{0}^{\infty} Q(x-y)^ph(x) dx = O(e^{-pR}\|h\|_{H^1}). 
\end{align*}
Indeed, by \eqref{ortho} and $y>2R$, we 
\begin{align*}
	\left| \re \int_{0}^{\infty} \mathcal{T}_yQ^p h dx\right|
	&=\left|  \re \int_{0}^{\infty} \chi_R^+ \mathcal{T}_yQ^p h dx 
	+ \re \int_{0}^{\infty} \chi_R^c \mathcal{T}_yQ^p h dx \right|
	\\
	&=\left| \re \int_{0}^{\infty} \chi_R^c \mathcal{T}_yQ^p h dx\right|
	\\
	&\leq \left( \int_{0}^{R} \mathcal{T}_yQ^{2p} dx \right)^{\frac{1}{2}} \|h\|_{L^2}
	\\
	&\lesssim  e^{-py+pR}\|h\|_{L^2} 
	\\
	&\lesssim  e^{-pR}\|h\|_{H^1}.
\end{align*}
Therefore, we obtain
\begin{align*}
	 -4 \rho \int_{0}^{\infty} \chi_R(x)Q(x-y)^{p+1} dx
	 = \mu_\gamma(u) 
	+ \|g\|_{H^1}^2
	+2\mathcal{Q}(y) +2\mathrm{Err} 
	+ O(e^{-pR}\|h\|_{H^1}). 
\end{align*}
By the Young inequality, \eqref{Q}, and \eqref{Q'}, we have
\begin{align*}
	0 \leq \mathrm{Err} \lesssim e_\gamma(y) + \|g\|_{H^1}^2.
\end{align*}
We also have $|\int_{0}^{\infty} \chi_R(x)Q(x-y)^{p+1} dx| \geq C$. 
Thus, by Corollary \ref{cor4.5}, we get
\begin{align*}
	|\rho| \lesssim |\mu_\gamma(u)|
	+ \|g\|_{H^1}^2
	+e_\gamma(y) 
	+o(e^{-(p+1)y})
	+ O(e^{-pR}\|h\|_{H^1}).
\end{align*}
The desired estimate is obtained. 
\end{proof}

%\begin{align*}
%	& -\int_{-\infty}^{0} |Q'(x-y)|^2dx- \int_{-\infty}^{0} |Q(x-y)|^2dx 
%	-2 Q'(-y)\re g(0) + \frac{|\gamma|}{2} |Q(-y)+g(0)|^2  \\
%	&= -\int_{-\infty}^{0} |Q'(x-y)|^2dx- \int_{-\infty}^{0} |Q(x-y)|^2dx + \frac{|\gamma|}{2} Q(-y)^2 \\
%	&\quad  +\{|\gamma|Q(-y)-2 Q'(-y)\}\re g(0)  + \frac{|\gamma|}{2}|g(0)|^2   \\
%	& =\mathcal{Q}(y)  +\frac{1}{2|\gamma|} \{|\gamma|Q(-y)-2 Q'(-y)\}^2
%	+\{|\gamma|Q(-y)-2 Q'(-y)\}\re g(0) + \frac{|\gamma|}{2}|g(0)|^2 
%\end{align*}
%where we set
%\begin{align*}
%	\mathcal{Q}(y)&:=-\int_{-\infty}^{0} |Q'(x-y)|^2dx
%	- \int_{-\infty}^{0} |Q(x-y)|^2dx  + \frac{|\gamma|}{2} Q(-y)^2 
%	-\frac{1}{2|\gamma|} \{|\gamma|Q(-y)-2 Q'(-y)\}^2
%\end{align*}
%As shown later, we have
%\begin{align*}
%	\mathcal{Q}(y)=\left( 1 -\frac{2}{|\gamma|}\right)c_p^2 e^{-2y} 
%	- \frac{4p(|\gamma|-2)}{|\gamma|(p-1)} c_p^2 e^{-(p+1)y}
%	+o( e^{-(p+1)y}). 
%\end{align*}
%We note that $\mathcal{Q}(y)=o( e^{-(p+1)y})$ if $\gamma=-2$. 
%We set
%\begin{align*}
%	\mathrm{Err}:=\frac{1}{2|\gamma|} \{|\gamma|Q(-y)-2 Q'(-y)\}^2
%	+\{|\gamma|Q(-y)-2 Q'(-y)\}\re g(0) + \frac{|\gamma|}{2}|g(0)|^2.
%\end{align*}
%By the Young inequality, we get $\mathrm{Err} \geq 0$. 
%If $\gamma <-2$, then we have
%\begin{align*}
%	\mathrm{Err} \lesssim e^{-2y} + \|g\|_{H^1}^2.
%\end{align*}
%If $\gamma=-2$, then we have
%\begin{align*}
%	\mathrm{Err} \lesssim e^{-(p+1)y} + \|g\|_{H^1}^2.
%\end{align*}

We next give an estimate using the action. 
By the mass-energy condition, we have
\begin{align*}
	0&=S_\gamma(u) -2S(Q)\\
	&=S(u)- S(Q(|\cdot|-y))+S(Q(|\cdot|-y)) - 2S(Q) + \frac{|\gamma|}{2} |u(0)|^2.
\end{align*}

We set 
\begin{align*}
	S_{(0,\infty)}(f):= \frac{1}{2}\|f'\|_{L^2(0,\infty)}^2 + \frac{1}{2}\|f\|_{L^2(0,\infty)}^2 -\frac{1}{p+1} \|f\|_{L^{p+1}(0,\infty)}^{p+1}
\end{align*}

\begin{lemma}
\label{lem4.10}
We have
\begin{align*}
	S(u) - S(Q(|\cdot|-y))
	= -2Q'(-y)\re g(0) + B(g,g) +o(\|g\|_{L^2}^2),
\end{align*}
where $B=B_y$ is defined in Section \ref{sec2.6}. 
%\begin{align*}
%	B(\varphi,\psi)&:=\re \int_{0}^{\infty} \varphi'(x)\overline{\psi'(x)} + \varphi(x)\overline{\psi(x)} dx  \\
%	&\quad -\int_{0}^{\infty} Q(x-y)^{p-1} (p\varphi_1(x)\psi_1(x) + \varphi_2(x)\psi_2(x)) dx.
%\end{align*}
%for $\varphi=\varphi_1+i \varphi_2 \in H_\even^1(\mathbb{R})$, $\psi=\psi_1+i \psi_2\in H_\even^1(\mathbb{R})$. 
\end{lemma}

\begin{proof}
By the symmetry and $u=e^{i\theta+it}(Q(|\cdot|-y) + g)$, we have
\begin{align*}
	S(u) - S(Q(|\cdot|-y))
	&=S(Q(|\cdot|-y) + g) - S(Q(|\cdot|-y))
	\\
	&=2\{S_{(0,\infty)}(Q(\cdot-y) + g) - S_{(0,\infty)}(Q(\cdot-y)) \}.
\end{align*}
By Taylor expansion, we get
\begin{align*}
	&S_{(0,\infty)}(Q(\cdot-y) + g) - S_{(0,\infty)}(Q(\cdot-y))  \\ 
	&=\langle S_{(0,\infty)}'(Q(\cdot-y), g \rangle +\frac{1}{2}\langle  S_{(0,\infty)}''(Q(\cdot-y))g,g \rangle +o(\|g\|_{L^2}^2).
\end{align*}
Now, by \eqref{ellip}, we have
\begin{align*}
	\langle S_{(0,\infty)}'(Q(\cdot-y)), g \rangle
	& = \re \int_{0}^{\infty} (-Q''+Q-Q^p)(x-y) \overline{g(x)}dx
	-\re \{Q'(-y) \overline{g(0)}\} \\
	&= -Q'(-y)\re g(0).
\end{align*}
Moreover, for $\varphi=\varphi_1+i \varphi_2 \in H_\even^1(\mathbb{R})$, $\psi=\psi_1+i \psi_2\in H_\even^1(\mathbb{R})$, we also have
\begin{align*}
	\langle  S_{(0,\infty)}''(Q(\cdot-y))\varphi, \psi \rangle
	&=\re \int_{0}^{\infty} \varphi'(x)\overline{\psi'(x)} + \varphi(x)\overline{\psi(x)} dx  \\
	&\quad -\int_{0}^{\infty} Q(x-y)^{p-1} (p\varphi_1(x)\psi_1(x) + \varphi_2(x)\psi_2(x)) dx
	\\
	&=B(\varphi,\psi).
\end{align*}
This gives the statement. 
\end{proof}

\begin{lemma}
\label{lem4.11}
We have
\begin{align*}
	&S(Q(|\cdot|-y)) - 2S(Q)
	\\
	&=- \int_{-\infty}^{0} |Q'(x-y)|^2 +|Q(x-y)|^2 dx + \frac{2}{p+1} \int_{-\infty}^{0} |Q(x-y)|^{p+1} dx.
\end{align*}
\end{lemma}

\begin{proof}
By the symmetry, we have
\begin{align*}
	S(Q(|\cdot|-y)) - 2S(Q)
	=2\{ S_{(0,\infty)}(Q(\cdot-y)) - S(Q)\}.
\end{align*}
For a function $f$ and $\alpha>0$, we have
\begin{align*}
	\int_{0}^{\infty} |f(x-y)|^\alpha dx = \|f\|_{L^\alpha}^\alpha - \int_{-\infty}^{0} |f(x-y)|^\alpha dx.
\end{align*}
Thus we get
\begin{align*}
	 &S_{(0,\infty)}(Q(\cdot-y)) - S(Q) \\
	 &=- \frac{1}{2}\int_{-\infty}^{0} |Q'(x-y)|^2 +|Q(x-y)|^2 dx + \frac{1}{p+1} \int_{-\infty}^{0} |Q(x-y)|^{p+1} dx.
\end{align*}
This gives us the statement. 
%By \eqref{}, we have
%\begin{align*}
%	&S(Q(|\cdot|-y)) - 2S(Q)\\
%	&=- \int_{-\infty}^{0} |Q'(x-y)|^2 +|Q(x-y)|^2 dx + \frac{2}{p+1} \int_{-\infty}^{0} |Q(x-y)|^{p+1} dx
%\end{align*}
\end{proof}

Combining Lemmas \ref{lem4.10} and \ref{lem4.11}, we get the following. 

\begin{corollary}
\label{cor4.12}
We have
\begin{align*}
	S_\gamma(u) -2S(Q)
	\geq \mathcal{Q}(y) + \frac{2}{p+1} \int_{-\infty}^{0} |Q(x-y)|^{p+1} dx
	+ B(g,g) +o(\|g\|_{L^2}^2).
\end{align*}
\end{corollary}

\begin{proof}
By Lemmas  \ref{lem4.10} and \ref{lem4.11}, we have
\begin{align*}
	S_\gamma(u) -2S(Q)
	&= S(u)- S(Q(|\cdot|-y))+S(Q(|\cdot|-y)) - 2S(Q) + \frac{|\gamma|}{2} |u(0)|^2
	\\
	&=
	- \int_{-\infty}^{0} |Q'(x-y)|^2 +|Q(x-y)|^2 dx
	-2Q'(-y)\re g(0) + \frac{|\gamma|}{2} |u(0)|^2
	\\
	&\quad + \frac{2}{p+1} \int_{-\infty}^{0} |Q(x-y)|^{p+1} dx+ B(g,g) +o(\|g\|_{L^2}^2) 
	\\
	&=\mathcal{Q}(y) + \mathrm{Err}  + \frac{2}{p+1} \int_{-\infty}^{0} |Q(x-y)|^{p+1} dx+ B(g,g) +o(\|g\|_{L^2}^2), 
\end{align*}
where $\mathrm{Err}$ is defined in the proof of Lemma \ref{lem4.9}. As shown in that proof,  $\mathrm{Err}\geq 0$ by the Young inequality. Thus, the desired estimate is obtained. 
\end{proof}

\begin{lemma}
\label{lem4.13}
There exist $c,C>0$ such that 
we have
\begin{align*}
	B(g,g) \geq c\|h\|_{H^1}^2 - C|\rho|^2
\end{align*}
for sufficiently large $R>0$.
\end{lemma}

\begin{proof}
Since we have $g= \rho \chi_R Q(|\cdot|-y) +h$, we get
\begin{align*}
	B(g,g) 
%	&= B( \rho \chi_R Q(|\cdot|-y) +h, \rho \chi_R Q(|\cdot|-y) +h)  \\
	=  |\rho|^2 B( \chi_R Q(|\cdot|-y), \chi_R Q(|\cdot|-y))
	+ 2\rho B(\chi_R Q(|\cdot|-y),h)
	+ B(h,h).
\end{align*}
We have
\begin{align}
\label{eq4.3'}
	|\rho|^2|B( \chi_R Q(|\cdot|-y), \chi_R Q(|\cdot|-y))| \leq C|\rho|^2, 
\end{align}
where $C$ is independent of $R$. 
\begin{comment}
Indeed, we have
\begin{align*}
	&|B( \chi_R Q(|\cdot|-y), \chi_R Q(|\cdot|-y))|\\
	&\leq \int_{0}^{\infty}  |\{\chi_R(x) Q(x-y)\}' |^2 + |\chi_R(x) Q(x-y)|^2dx \\
	&\quad + p \int_{0}^{\infty} Q(x-y)^p |\chi_R(x) Q(x-y)|^2  dx\\
	&\leq \int_{\mathbb{R}}  \frac{1}{R^2}\left|\chi'\left(\frac{x}{R}\right) Q(x-y) \right|^2 + |Q'(x-y)|^2 + |Q(x-y)|^2dx \\
	&\quad + p \int_{\mathbb{R}} Q(x-y)^{p+1}  dx\\
	&\leq C.
\end{align*}
\end{comment}
By the Young inequality, we also have
\begin{align}
\label{eq4.4'}
	|2\rho B(\chi_R Q(|\cdot|-y),h)|
	\leq C|\rho|\|h\|_{H^1}
	\leq \varepsilon\|h\|_{H^1}^2 +C_\varepsilon |\rho|^2,
\end{align}
where $\varepsilon>0$ is a sufficiently small number.

We will estimate $B(h,h)$. 
\begin{align*}
	B(h,h)&=B(\chi_R h + \chi_R^c h, \chi_R h + \chi_R^c h)\\
	&=B(\chi_R h, \chi_R h ) 
	+2B(\chi_R h, \chi_R^c h)
	+B(\chi_R^c h, \chi_R^c h)
\end{align*}
First, we consider $B(\chi_R h, \chi_R h )$. By a direct calculation, we have
\begin{align*}
	B(\chi_R h, \chi_R h ) 
%	&= \re \int_{0}^{\infty} |(\chi_R h)'(x)|^2 + \chi_R(x)^2 |h(x)|^2 dx  \\
%	&\quad -\int_{0}^{\infty} Q(x-y)^{p-1} (p\chi_R(x)^2 h_1(x)^2+ \chi_R(x)^2 h_2(x)^2) dx \\
	&= \re \int_{\mathbb{R}} |(\chi_R^+ h)'(x)|^2 + \chi_R^+(x)^2 |h(x)|^2 dx  \\
	&\quad -\int_{\mathbb{R}} Q(x-y)^{p-1} (p\chi_R^+(x)^2 h_1(x)^2+ \chi_R^+(x)^2 h_2(x)^2) dx\\
%	&= \re \int_{\mathbb{R}} |\{(\chi_R^+ h)(x+y)\}'|^2 + \chi_R^+(x+y)^2 |h(x+y)|^2 dx  \\
%	&\quad -\int_{\mathbb{R}} Q(x)^{p-1} (p\chi_R^+(x+y)^2 h_1(x+y)^2+ \chi_R^+(x+y)^2 h_2(x+y)^2) dx\\
	&=\Phi( \mathcal{T}_{-y}(\chi_R^+ h)). 
\end{align*}
Since we have $\Phi( \mathcal{T}_{-y}(\chi_R^+ h)) \geq c \| \chi_R h\|_{H^1}^2 - CR^{-1}\|h\|_{H^1}^2$ by the coercivity Lemma \ref{lem2.12}, we have
\begin{align*}
	B(\chi_R h, \chi_R h )  \geq c \| \chi_R h\|_{H^1}^2 - \frac{C}{R}\|h\|_{H^1}^2. 
\end{align*}
Next we consider $B(\chi_R h, \chi_R^c h)$. 
Now we have
\begin{align*}
	\int_{0}^{\infty} (\chi_R h)'(x)\overline{(\chi_R^c h)'(x)}dx 
%	&= \int_{0}^{\infty} (\chi_R'(x) h(x) + \chi_R(x) h'(x)) \overline{(\chi_R^c)'(x)h(x) + \chi_R^c(x) h'(x)}dx \\
%	&= \int_{0}^{\infty} \chi_R(x) \chi_R^c(x)  |h'(x)|^2dx \\
%	&+ \int_{0}^{\infty} \chi_R'(x)(\chi_R^c)'(x) |h(x)|^2 dx \\
%	&+ \int_{0}^{\infty} \chi_R'(x)\chi_R^c(x) h(x)\overline{h'(x)}dx \\
%	&+ \int_{0}^{\infty} \chi_R(x) (\chi_R^c)'(x)h'(x)\overline{h(x)}dx\\
	= \int_{0}^{\infty} \chi_R(x) \chi_R^c(x)  |h'(x)|^2dx + O(R^{-1}\|h\|_{H^1}^2).
\end{align*}
and 
\begin{align*}
	\left|\int_{0}^{\infty} \mathcal{T}_yQ^{p-1} (p\chi_R\chi_R^c |h_1|^2 + \chi_R\chi_R^c |h_2|^2) dx\right|
	&\leq C \left| \int_{0}^{R}  \mathcal{T}_yQ^{p-1} dx\right| \|h\|_{L^\infty}^2
%	\\
%	&\leq C e^{-(p-1)y} \int_{0}^{R} e^{(p-1)x}dx \|h\|_{H^1}^2
	\\
	&\leq C e^{(p-1)R} e^{-(p-1)y} \|h\|_{H^1}^2
	\\
	&\leq C e^{-(p-1)R} \|h\|_{H^1}^2
\end{align*}
Thus we get
\begin{align*}
	B(\chi_R h, \chi_R^c h)
	&= \re \int_{0}^{\infty} (\chi_R h)'(x)\overline{(\chi_R^c h)'(x)} + \chi_R(x) h(x)\overline{\chi_R^c(x) h(x)} dx  \\
	&\quad -\int_{0}^{\infty} Q(x-y)^{p-1} (p\chi_R(x)\chi_R^c(x) |h_1(x)|^2 + \chi_R(x)\chi_R^c |h_2(x)|^2) dx \\
	&=\int_{0}^{\infty} \chi_R(x) \chi_R^c(x) ( |h'(x)|^2+|h(x)|^2) dx +O(R^{-1}\|h\|_{H^1}^2)+O( e^{-(p-1)R} \|h\|_{H^1}^2)
	\\
	&=\frac{1}{2}\int_{\mathbb{R}} \chi_R(x) \chi_R^c(x) ( |h'(x)|^2+|h(x)|^2) dx +O(R^{-1}\|h\|_{H^1}^2)+O( e^{-(p-1)R} \|h\|_{H^1}^2)
\end{align*}
%Thus we get
%\begin{align*}
%	B(\chi_R h, \chi_R^c h) 
%	= \int_{0}^{\infty} \chi_R(x) \chi_R^c(x) ( |h'(x)|^2+|h(x)|^2) dx +O(R^{-1}\|h\|_{H^1}^2)+O( e^{-(p-1)R} \|h\|_{H^1}^2)
%\end{align*}
In a similar way, we have
\begin{align*}
	B(\chi_R^c h, \chi_R^c h) 
	=\frac{1}{2} \int_{\mathbb{R}} \chi_R^c(x) \chi_R^c(x)   ( |h'(x)|^2+|h(x)|^2) dx +O(R^{-1}\|h\|_{H^1}^2)+O( e^{-(p-1)R} \|h\|_{H^1}^2).
\end{align*}
Therefore, we obtain
\begin{align*}
	B(h,h) &\geq  c \| \chi_R h\|_{H^1}^2 - \frac{C}{R}\|h\|_{H^1}^2 \\
	&+ \int_{\mathbb{R}}\chi_R(x) \chi_R^c(x) ( |h'(x)|^2+|h(x)|^2) dx \\
	&+\frac{1}{2}\int_{\mathbb{R}} \chi_R^c(x) \chi_R^c(x)   ( |h'(x)|^2+|h(x)|^2) dx  \\
	&+O(R^{-1}\|h\|_{H^1}^2)+O( e^{-(p-1)R} \|h\|_{H^1}^2).
\end{align*}
Here, it holds that
\begin{align*}
	 c \| \chi_R h\|_{H^1}^2 + \int_{\mathbb{R}}\{ \chi_R \chi_R^c+\frac{1}{2}(\chi_R^c)^2\} ( |h'|^2+|h|^2) dx
	 \geq c\|h\|_{H^1}^2 -\frac{C}{R}\|h\|_{H^1}^2.
\end{align*}
\begin{comment}
Indeed, we have
\begin{align*}
	\| \chi_R h\|_{H^1}^2
	&= \int_{\mathbb{R}} |(\chi_Rh)'(x)|^2 + |(\chi_Rh(x))|^2dx \\
	&= \int_{\mathbb{R}} \chi_R(x)^2 (|h'(x) |^2+ |h(x)|^2)dx + O(R^{-1}\|h\|_{H^1}^2)
\end{align*}
and thus
\begin{align*}
	&c \| \chi_R h\|_{H^1}^2 + \int_{\mathbb{R}}\{ \chi_R \chi_R^c+\frac{1}{2}(\chi_R^c)^2\} ( |h'|^2+|h|^2) dx
	\\ 
	&\geq c \left\{ \int_{\mathbb{R}} \chi_R^2 (|h' |^2+ |h|^2)dx + \int_{\mathbb{R}}( 2\chi_R \chi_R^c +(\chi_R^c)^2) ( |h'|^2+|h|^2) dx \right\} 
	+ O(R^{-1}\|h\|_{H^1}^2)\\ 
	&\geq c \left\{ \int_{\mathbb{R}} (\chi_R+\chi_R^c)^2 (|h' |^2+ |h|^2) dx \right\} -\frac{C}{R}\|h\|_{H^1}^2 \\ 
	&\geq c \|h\|_{H^1}^2
	-\frac{C}{R}\|h\|_{H^1}^2 \\ 
\end{align*}
\end{comment}
This implies 
\begin{align}
\label{eq4.5}
	B(h,h) &\geq c\|h\|_{H^1}^2 -\frac{C}{R}\|h\|_{H^1}^2  +O(R^{-1}\|h\|_{H^1}^2)+O( e^{-(p-1)R} \|h\|_{H^1}^2)
	\\
	&\geq c\|h\|_{H^1}^2 -\frac{C}{R}\|h\|_{H^1}^2.
\end{align}
By combining \eqref{eq4.3}--\eqref{eq4.5}, we get
\begin{align*}
	B(g,g) \geq c\|h\|_{H^1}^2 -\frac{C}{R}\|h\|_{H^1}^2 -C|\rho|^2 -\varepsilon\|h\|_{H^1}^2 -C_\varepsilon |\rho|^2.
\end{align*}
Taking $\varepsilon>0$ sufficiently small and $R>0$ sufficiently large, we obtain
\begin{align*}
	B(g,g) \geq c\|h\|_{H^1}^2 -C|\rho|^2. 
\end{align*}
This completes the proof. 
\end{proof}
\begin{lemma}
\label{lem4.14}
We have
\begin{align*}
	e_\gamma(y) + \|h\|_{H^1}^2 \lesssim |\rho|^2.
\end{align*}
\end{lemma}

\begin{proof}
By Corollaries \ref{cor4.5} and \ref{cor4.12}  and Lemmas \ref{lem4.6} and \ref{lem4.13}, we get the desired estimate.
\end{proof}

\begin{corollary}
\label{cor4.15}
It holds that 
\begin{align*}
	|\rho| \lesssim |\mu_\gamma(u)|. 
\end{align*}
In particular, we have
\begin{align*}
	e_\gamma(y) + \|h\|_{H^1}^2 \lesssim |\mu_\gamma(u)|^2. 
\end{align*}
\end{corollary}

\begin{proof}
By Lemmas \ref{lem4.7}, \ref{lem4.9}, and \ref{lem4.14}, we have
\begin{align*}
	|\rho| \lesssim |\mu_\gamma(u)|
	+ \|g\|_{H^1}^2
	+e_\gamma(y) 
	+e^{-R}\|h\|_{H^1}
	\lesssim |\mu_\gamma(u)|
	+|\rho|^2 
	+e^{-R}|\rho|. 
\end{align*}
This shows that 
\begin{align*}
	|\rho| \lesssim |\mu_\gamma(u)|
\end{align*}
by taking $R$ large and $|\rho|$ small. The second estimate follows from this and Lemma \ref{lem4.14}. 
\end{proof}

\subsection{Estimates for time derivatives of modulation parameters}

We have estimates for the time derivatives of the modulation parameters as follows. 
\begin{lemma}
\label{lem4.16}
We have
\begin{align*}
	|\dot{y}| + |\dot{\rho}| + |\dot{\theta}|
	\lesssim |\mu_\gamma(u)|. 
\end{align*}
\end{lemma}

\begin{proof}
By a direct calculation, we have 
\begin{align*}
	&i\dot{h}+h'' +\gamma \delta h
	\\
	&=\dot{\theta}Q(|\cdot|-y)
	+\dot{\theta}\rho \chi_R Q(|\cdot|-y)
	+(\dot{\theta}+1)h
	+\gamma_y \delta Q(|\cdot|-y) \\
	&\quad +|Q(|\cdot|-y)|^{p-1}Q(|\cdot|-y)-|Q(|\cdot|-y)+g|^{p-1}(Q(|\cdot|-y)+g) \\
	&\quad -i \dot{\rho} \chi_R Q(|\cdot|-y) +i \dot{y}(1+\rho \chi_R)Q'(|\cdot|-y) \\ 
	&\quad  - \rho \chi_R'' Q(|\cdot|-y)
	-2\rho \frac{x}{|x|}\chi_R' Q'(|\cdot|-y)
	+\rho \chi_R Q(|\cdot|-y)^p
	- \gamma \delta Q(|\cdot|-y), 
\end{align*}
where $\gamma_y:= -2Q'(-y)/Q(-y)$ and we used the fact that $Q(|\cdot|-y)$ satisfies the equation 
\begin{align*}
	-\{Q(|\cdot|-y)\}'' + Q(|\cdot|-y)
	= \gamma_y \delta Q(|\cdot|-y) +|Q(|\cdot|-y)|^{p-1}Q(|\cdot|-y). 
\end{align*}
We have delta interactions in the equation of $h$. However, we take couplings with the equation and functions including $\chi_R^+$. Thus they  do not appear in the following calculations. 
By testing the equation of $h$ with $\chi_R^+iQ(\cdot-y)$, the orthogonality condition implies that
\begin{align}\label{eq44}
	|\dot{\theta}| 
	\lesssim |\rho|+\|h\|_{H^1}+\|g\|_{H^1} + \|h\|_{H^1}|\dot{y}|
	\lesssim |\mu_\gamma(u)| + \|h\|_{H^1}|\dot{y}|,
\end{align}
where we also use Lemma \ref{lem4.7} and Corollary \ref{cor4.15}. 
Testing the equation with $i\chi_R^+ Q(\cdot -y)^{p}$ and using the orthogonality, we get
\begin{align}
	|\dot{\rho}| 
	&\lesssim (|\dot{\theta}|+1)\|h\|_{H^1} + \|g\|_{H^1} + (1+|\rho|)e^{-R}|\dot{y}|+|\dot{y}| \|h\|_{H^1} 
	\\
	&\lesssim |\mu_\gamma(u)| %+e^{-R}|\dot{y}|
	+(e^{-R}+ \|h\|_{H^1})|\dot{y}|. \label{eq4.6}
\end{align}
Testing the equation with $i\chi_R^+ Q'(\cdot-y)$ and using the orthogonality, we obtain
\begin{align}
\label{eq4.7}
	|\dot{y}| \lesssim |\rho| + \|h\|_{H^1} + \|g\|_{H^1} + e^{-R}|\dot{\rho}|
	+ \|h\|_{H^1}|\dot{y}|
	\lesssim |\mu_\gamma(u)| + e^{-R}|\dot{\rho}|
	+ \|h\|_{H^1}|\dot{y}|.
\end{align}
These estimates \eqref{eq44}\noeqref{eq4.6}--\eqref{eq4.7} 
%These two estimates \eqref{eq4.6} and \eqref{eq4.7}
imply that
\begin{align*}
	|\dot{\theta}| +|\dot{y}| + |\dot{\rho}|
	\lesssim |\mu_\gamma(u)| + (e^{-R}+ \|h\|_{H^1})|\dot{y}| + e^{-R}|\dot{\rho}|
	%|\dot{y}| + |\dot{\rho}|\lesssim |\mu_\gamma(u)| + e^{-R}(|\dot{y}| + |\dot{\rho}|)
\end{align*}
By taking $R$ large and $\|h\|_{H^1}$ small, we get
\begin{align*}
	|\dot{\theta}| + |\dot{y}| + |\dot{\rho}|
	\lesssim |\mu_\gamma(u)|. 
\end{align*}
The proof is complete. 
\end{proof}

%%%%%%%%%%%%%%%%%%%%%%%%%%%%%%%%%%%%%%%%%%

\section{Proof of scattering}

The proof is very similar to \cite{GuIn22}, so we only give an outline here. 

Suppose that statement (1) of Theorem \ref{thm2.5} does not hold. Then there exists a global solution $u \in C(\mathbb{R}: H_\even^1(\mathbb{R}))$ with 
\begin{align*}
	E_\gamma(u)=2E(Q),~M(u)=2M(Q),~K_\gamma(u(t)) >0,
\end{align*}
and 
\begin{align*}
	\|u\|_{S(\mathbb{R})}=\infty.
\end{align*}
We may assume that $\|u\|_{L_t^aL_x^r((0,\infty)\times \mathbb{R})}=\infty$. We call this solution a (forward) critical element. 

\begin{lemma}[Compactness of a critical element]
Let $u \in C(\mathbb{R}: H_\even^1(\mathbb{R}))$ be a solution with
\begin{align*}
	E_\gamma(u)=2E(Q),~M(u)=2M(Q),~K_\gamma(u(t)) >0,
\end{align*}
and
\begin{align*}
	\|u\|_{S(0,\infty)}=\infty.
\end{align*}
Then there exists a function $x:[0,\infty) \to [0,\infty)$ such that for any $\varepsilon>0$ there exists $R=R(\varepsilon)>0$ such that 
\begin{align*}
	\int_{\{|x-x(t)|>R\} \cap \{|x+x(t)|>R\}} |u'(t,x)|^2 + |u(t,x)|^2 dx <\varepsilon
\end{align*}
for any $t \in [0,\infty)$. 
\end{lemma}

\begin{proof}
The proof is based on linear profile decomposition and long time perturbation. See  \cite{BaVi16,IkIn17} for these. 
The statement can be shown in a similar way to \cite[Proposition 35]{GuIn22}. We omit the proof. 
\end{proof}

The compactness also holds for 
\begin{align}
\label{eq5.1}
	X(t):=
	\begin{cases}
	x(t) &(t \in [0,\infty)\setminus I_{\mu_0}),
	\\
	y(t) & (t \in I_{\mu_0}).
	\end{cases}
\end{align}
See \cite[Lemma 38]{GuIn22}. That is, for any $\varepsilon>0$ there exists $R=R(\varepsilon)>0$ such that 
\begin{align*}
	\int_{\{|x-X(t)|>R\} \cap \{|x+X(t)|>R\}} |u'(t,x)|^2 + |u(t,x)|^2 dx <\varepsilon
\end{align*}
for all $t \in [0,\infty)$. 

\begin{lemma}
Let $\{t_n\}$ be an arbitrary time sequence in $[0,\infty)$. We have the following.
\begin{enumerate}[(1)]
\item If $|X(t_n)|$ is unbounded, then by taking a subsequence of $\{t_n\}$, still denoted by $\{t_n\}$, there is $\psi \in H^1(\mathbb{R})$ such that $u(t_n) - (\psi(\cdot-X(t_n))+\psi(-\cdot-X(t_n))) \to 0$ in $H^1(\mathbb{R})$. 
\item If $|X(t_n)|$ is bounded, then by taking a subsequence of $\{t_n\}$, still denoted by $\{t_n\}$, there is $\psi \in H^1(\mathbb{R})$ such that $u(t_n) \to \psi$ in $H^1(\mathbb{R})$.
\end{enumerate}
\end{lemma}

\begin{proof}
See Lemma 39 in \cite{GuIn22}. 
\end{proof}

%%%%%%%%%%%%%%%%%%%%%%%%%%%%%%%%%%%%%%%%%%
\subsection{Elimination of the critical element}

\subsubsection{$X(t)$ is bounded}

\begin{lemma}
\label{lem5.3}
Let $u \in C(\mathbb{R}: H_\even^1(\mathbb{R}))$ be a solution with
\begin{align*}
	E_\gamma(u)=2E(Q),~M(u)=2M(Q),~K_\gamma(u(t)) >0. 
\end{align*}
Assume that there exists a function $x:[0,\infty) \to [0,\infty)$ such that for any $\varepsilon>0$ there exists $R=R(\varepsilon)>0$ such that 
\begin{align*}
	\int_{\{|x-X(t)|>R\} \cap \{|x+X(t)|>R\}} |u'(t,x)|^2 + |u(t,x)|^2 dx <\varepsilon
\end{align*}
for any $t \in [0,\infty)$, where $X$ is defined in \eqref{eq5.1}. 
Then $X$ is bounded. 
\end{lemma}

To show this, we prepare the following lemmas. 

\begin{lemma}
There exists $C_\varepsilon>0$ such that 
\begin{align*}
	\int_{t_1}^{t_2} \mu_\gamma(t) dt \leq C_\varepsilon (1+\sup_{t \in [t_1,t_2]}|X(t)|)(\mu_\gamma(t_1)+\mu_\gamma(t_2))
\end{align*}
for any $t_2>t_1 >0$. 
\end{lemma}

\begin{proof}
We use Lemma \ref{lem2.10} in this proof. 
See \cite[Lemma 41]{GuIn22} 
\end{proof}

\begin{lemma}
\label{lem5.5}
Let $\{t_n\}$ be a sequence such that $t_n \to \infty$ as $n \to \infty$. Then $X(t_n) \to  \infty$ if and only if $\mu_\gamma(t_n) \to 0$. 
\end{lemma}

\begin{proof}
See \cite[Lemma 42]{GuIn22}. Here, we use Corollary~\ref{cor4.15}. 
\end{proof}

\begin{lemma}
There exists $C>0$ such that
\begin{align*}
	|X(t)-X(s)| \leq C
\end{align*}
for all $t,s\geq 0$ satisfying $|t-s|\leq 1$. 
\end{lemma}

\begin{proof}
See \cite[Lemma 43]{GuIn22}. 
\end{proof}

\begin{lemma}
There exists a constant $C>0$ such that
\begin{align*}
	|X(t_1)-X(t_2)| \leq C \int_{t_1}^{t_2} \mu_\gamma (t)dt
\end{align*}
for all $t_1,t_2\geq 0$ satisfying $t_1+1\leq t_2$.
\end{lemma}

\begin{proof}
We use the estimate of $\dot{y}$ in Lemma \ref{lem4.16} to show this statement. 
See \cite[Lemma 44]{GuIn22}. 
\end{proof}

Combining these lemmas, we get Lemma \ref{lem5.3}. See the proof of Proposition 40 in \cite{GuIn22}.

%%%%%%%%%%%%%%%%%%%%%%%%%%%%%%%%%%%%%%%%%%

\subsubsection{Contradiction if $X(t)$ is bounded}

Since $X$ is bounded, it holds that for any $\varepsilon>0$ there exists $R=R(\varepsilon)>0$ such that 
\begin{align*}
	\int_{|x|>R} |u'(t,x)|^2 + |u(t,x)|^2 dx <\varepsilon
\end{align*}
for all $t \in [0,\infty)$. 

\begin{lemma}
We have
\begin{align*}
	\lim_{T \to \infty} \frac{1}{T} \int_{0}^{T} \mu_\gamma(t)dt =0.
\end{align*}
\end{lemma}

\begin{proof}
See \cite[Lemma 45]{GuIn22}. 
\end{proof}

\begin{corollary}
There exists a time sequence $\{t_n\}$ such that $\lim_{n \to \infty} \mu_\gamma(t_n)=0$. 
\end{corollary}

\begin{proof}
See \cite[Corollary 46]{GuIn22}. 
\end{proof}

Let $\{t_n\}$ be a time sequence such that $\lim_{n \to \infty} \mu_\gamma(t_n)=0$. Then $X(t_n)$ must diverge by Lemma \ref{lem5.5}. This is a contradiction to the boundedness of $X$. The conclusion is Theorem \ref{thm2.5}. 

%%%%%%%%%%%%%%%%%%%%%%%%%%%%%%%%%%%%%%%%%%

\section{Proof of blow-up}

The proof of Statement (2) of Theorem \ref{thm2.5}is similar to \cite{GuIn22}, and so we omit details.

\begin{lemma}
\label{lem6.1}
Let $\varphi \in C_\even^1(\mathbb{R})$ be a real-valued function with $\varphi(0)=0$ and $f \in H_\even^1(\mathbb{R})$. We assume that they satisfy $\int_{\mathbb{R}}|\varphi'|^2|f|^2dx < \infty$, $M(f)=2M(Q)$, and $E_\gamma(f)=2E(Q)$. Then it holds that 
\begin{align*}
	\left| \im \int_{\mathbb{R}} \varphi '(x) \partial_x f(x)\overline{f(x)}dx \right| \lesssim \mu_\gamma(f)^2 \int_{\mathbb{R}} |\varphi '(x)|^2|f(x)|^2dx.
\end{align*}
\end{lemma}

\begin{proof}
This can be shown in a similar way to \cite[Lemma 47]{GuIn22} by using the Gagliardo--Nirenberg inequality, Lemma \ref{lem2.4}. See also \cite[Lemma 4.13]{GuIn22p} We omit the proof. 
\end{proof}

\begin{corollary}
Let $u_0 \in H_\even^1(\mathbb{R})$ satisfy $K_\gamma(u_0)<0$, $xu_0 \in L^2(\mathbb{R})$, $M(u_0)=2M(Q)$, and $E_\gamma(u_0)=2E(Q)$. Then we have
\begin{align*}
	\left| \im \int_{\mathbb{R}} x u'(t,x)\overline{u(t,x)}dx \right| \lesssim |K_\gamma(u(t))|^2 \int_{\mathbb{R}} |x|^2|u(t,x)|^2dx.
\end{align*}
\end{corollary}

\begin{proof}
This follows from Lemma \ref{lem6.1} with $\varphi(x)=x^2$ and $\mu_\gamma(u) \lesssim |K_\gamma(u(t))|$, which is shown in Proposition \ref{prop2.7}. 
\end{proof}

\begin{lemma}\label{lem6.3}
Let $u$ satisfy the conditions in Theorem \ref{thm2.5} (2). Assume that $u$ is global in positive time. Then it holds that
\begin{align*}
	\im \int_{\mathbb{R}} x u'(t,x) \overline{u(t,x)}dx >0
\end{align*}
for any $t$. Furthermore, there exists $c>0$ such that 
\begin{align*}
	\int_{t}^{\infty} |\mu_\gamma(s)| ds \lesssim e^{-ct}
\end{align*}
for all $t>0$. 
\end{lemma}

\begin{proof}
See \cite[Proposition 49]{GuIn22}. 
\end{proof}

\begin{corollary}
Under the same assumption in Theorem \ref{thm2.5} (2), $u$ blows up in negative time. 
\end{corollary}

\begin{proof}
See \cite[Corollary 50]{GuIn22}. 
\end{proof}

As a consequence, we get Theorem \ref{thm2.5} (see the proof of Proposition 14 (2) in \cite{GuIn22}). Here we remark that we use the estimates of $\dot{y}$ and $\dot{\rho}$ in Lemma \ref{lem4.16}, which imply $y$ converges to a constant by combining with Lemma \ref{lem6.3}, and thus the convergence contradicts the estimate of $e_\gamma(y)$ in Corollary \ref{cor4.15}.

\section*{Declarations}

The authors declare no conflicts of interest. Both authors have contributed equally to this paper. There is no available data related to this work. 

\section*{Acknowledgement}

The first author is partially supported by an NSERC Discovery Grant. 
The second author is partially supported by JSPS KAKENHI Grant-in-Aid for Early-Career Scientists No. JP18K13444. He also expresses appreciation toward JSPS Overseas Research Fellowship and the University of British Columbia for stay by the fellowship. 

%%%%%%%%%%%%%%%%%%%%%%%%%%%%%%%%%%%%%%%%%

\appendix

\section{Sign condition}
\label{secB}

We show the following in this appendix. 

\begin{proposition}
\label{propB.1}
If $f \in H_\even^1(\mathbb{R})$ satisfies $E_\gamma(f)M(f)^{\frac{1-s_c}{s_c}} \leq 2^{\frac{1}{s_c}}E(Q)M(Q)^{\frac{1-s_c}{s_c}}$ and $M(f) \geq 2M(Q_{\gamma^2/4,0})$, then the following are equivalent. 
\begin{enumerate}[(1)]
\item $K_\gamma(f)>0$.
\item $\sqrt{2}\|Q\|_{L^2}^{1-s_c}\|Q\|_{\dot{H}^1}^{s_c}>\|f\|_{L^2}^{1-s_c}\|f\|_{\dot{H}_\gamma^1}^{s_c}$. 
\end{enumerate}
\end{proposition}

\begin{lemma}
\label{lemB.2}
If $f \in H_\even^1(\mathbb{R})$ satisfies $E_\gamma(f)M(f)^{\frac{1-s_c}{s_c}} \leq 2^{\frac{1}{s_c}}E(Q)M(Q)^{\frac{1-s_c}{s_c}}$ and $M(f) \geq 2M(Q_{\gamma^2/4,0})$, then there exists $\omega \in (0,\gamma^2/4]$ such that $M(f)=2M(Q_{\omega,0})$ and $E_{\gamma}(f) \leq 2 E(Q_{\omega,0})$. 
\end{lemma}

\begin{proof}
By the monotonicity of $2M(Q_{\omega,0})$ for $\omega$, there exists a unique $\omega \in (0,\gamma^2/4]$ such that $M(f)=2M(Q_{\omega,0})$. Then the mass-energy condition implies $E_{\gamma}(f) \leq 2 E(Q_{\omega,0})$.
\end{proof}

\begin{proof}[Proof of Proposition \ref{propB.1}]
By Lemma \ref{lemB.2}, we have $\omega \in (0,\gamma^2/4]$ such that $M(f)=2M(Q_{\omega,0})$ and $E_{\gamma}(f) \leq 2 E(Q_{\omega,0})$. 
First we show $\sqrt{2}\|Q\|_{L^2}^{1-s_c}\|Q\|_{\dot{H}^1}^{s_c}>\|f\|_{L^2}^{1-s_c}\|f\|_{\dot{H}_\gamma^1}^{s_c}$ if $K_\gamma(f)>0$. Then, by Proposition 2.15 in \cite{IkIn17} and Lemma \ref{lem2.2.0}, we see that $I_{\omega,\gamma}(f)>0$. By the similar calculation to Lemma \ref{lem2.6}, we get
\begin{align*}
	0< I_{\omega,\gamma}(f) \leq \mu_{\omega,\gamma}(f) :=2\|Q_{\omega,0}'\|_{L^2}^2 - \|f\|_{\dot{H}_\gamma^1}^2  
\end{align*}
since we have $M(f)=2M(Q_{\omega,0})$, $E_{\gamma}(f) \leq 2 E(Q_{\omega,0})$, and $I_{\omega,0}(Q_{\omega,0})=0$. This and $M(f)=2M(Q_{\omega,0})$ gives us that 
\begin{align*}
	\|f\|_{L^2}^{1-s_c}\|f\|_{\dot{H}_\gamma^1}^{s_c}< 
	\sqrt{2}\|Q_{\omega,0}\|_{L^2}^{1-s_c}\|Q_{\omega,0}\|_{\dot{H}^1}^{s_c} 
	=\sqrt{2}\|Q\|_{L^2}^{1-s_c}\|Q\|_{\dot{H}^1}^{s_c}
\end{align*}
by the scaling structure of $Q_{\omega,0}$. Next, we prove that $\sqrt{2}\|Q\|_{L^2}^{1-s_c}\|Q\|_{\dot{H}^1}^{s_c}>\|f\|_{L^2}^{1-s_c}\|f\|_{\dot{H}_\gamma^1}^{s_c}$ if $K_\gamma(f)<0$. Then we see that $I_{\omega,\gamma}(f)<0$ by Proposition 2.15 in \cite{IkIn17} and Lemma \ref{lem2.2.0}. By the Gagliardo--Nirenberg type inequality, Lemma \ref{lem2.4}, we get 
\begin{align*}
	0 > I_{\omega,\gamma}(f) 
	\geq \|f\|_{\dot{H}_{\omega,\gamma}^1}^2 - (2^{-\frac{p-1}{p+1}}C_{\omega,0})^{\frac{p+1}{2}}\|f\|_{\dot{H}_{\omega,\gamma}^1}^{p+1}.
\end{align*}
Thus we have
\begin{align*}
	\|f\|_{\dot{H}_{\omega,\gamma}^1}^{p-1} > 2^{\frac{p-1}{2}} C_{\omega,0}^{-\frac{p+1}{2}} = 2^{\frac{p-1}{2}} \left\{ \frac{2(p+1)}{p-1}S_{\omega,0}(Q_{\omega,0}) \right\}^{\frac{p-1}{2}}
\end{align*}
and thus 
\begin{align*}
	\|f\|_{\dot{H}_{\omega,\gamma}^1}^{2} 
	>  2 \left\{ \frac{2(p+1)}{p-1}S_{\omega,0}(Q_{\omega,0}) \right\}
	= 2 \omega^{1-s_c} \frac{2(p+1)}{p-1}S_{1,0}(Q). 
\end{align*}
Substituting $\omega=(M(Q)/M(f))^{1/s_c}$, which comes from $M(f)=M(Q_{\omega,0})=\omega^{-s_c}M(Q)$ into this, we get
\begin{align*}
	\|f\|_{\dot{H}_{\gamma}^1}^2 + \left(  \frac{M(Q)}{M(f)}\right)^{\frac{1}{s_c}} M(f) > 2 \left(  \frac{M(Q)}{M(f)}\right)^{\frac{1-s_c}{s_c}}\frac{2(p+1)}{p-1}S_{1,0}(Q). 
\end{align*}
By a direct calculation and the Pohozaev identities for $Q$, we obtain 
\begin{align*}
	\|f\|_{\dot{H}_{\gamma}^1}^2 M(f)^{\frac{1-s_c}{s_c}}
	>2 M(Q)^{\frac{1-s_c}{s_c}} \|Q'\|_{L^2}^2.
\end{align*}
This implies $\sqrt{2}\|Q\|_{L^2}^{1-s_c}\|Q\|_{\dot{H}^1}^{s_c}<\|f\|_{L^2}^{1-s_c}\|f\|_{\dot{H}_\gamma^1}^{s_c}$. At last, we note that $\sqrt{2}\|Q\|_{L^2}^{1-s_c}\|Q\|_{\dot{H}^1}^{s_c}=\|f\|_{L^2}^{1-s_c}\|f\|_{\dot{H}_\gamma^1}^{s_c}$ does not occur under the assumption. This follows from the above argument and the fact that $I_{\omega,\gamma}(f)=0$ does not occur. 
\end{proof}

\begin{remark}
By the similar argument to the above, we can show that $K_\gamma(f)>0$ is equivalent to $\|Q_{\omega,\gamma}\|_{L^2}^{1-s_c}\|Q_{\omega,\gamma}\|_{\dot{H}_\gamma^1}^{s_c}>\|f\|_{L^2}^{1-s_c}\|f\|_{\dot{H}_\gamma^1}^{s_c}$ under the assumptions $S_{\omega,\gamma}(f) \leq r_{\omega,\gamma}$ and $M(f)=M(Q_{\omega,\gamma})$ for some $\omega>\gamma^2/4$. However, it is difficult to remove $\omega$ from the condition unlike the low frequency case. In this sense, $K_\gamma(f)$ is more useful. 
\end{remark}

%%%%%%%%%%%%%%%%%%%%%%%%%%%%%%%%%%%%%%%%%

\section{Concluding remark}
\label{secA}

\begin{figure}[htb]
\centering
\begin{tikzpicture}[scale=2,samples=200]
\draw[->,>=stealth,semithick](-0.5,0)--(4,0) node[below]{$M$};%x軸
\draw[->,>=stealth,semithick](0,-0.5)--(0,3)node[left]{$E_\gamma$};%y軸
\draw(0,0)node[below left]{0};%原点
\draw[thick,dashed,domain=0.7:4] plot(\x, {2/(\x)});
\draw[thick,dotted,domain=0.35:4] plot(\x, {1/(\x)});
\draw[thick,domain=0.4:2] plot(\x, {(-1/4)*(\x)^2+(3/4)*(\x)+(\x)^(-1)});
\draw (0.7,20/7) node[above]{\tiny$\omega \to \infty$};
\draw (4,1/2) node[below]{\tiny$\omega \to 0$};
\draw (2,1) node[above right]{\tiny$\omega = \frac{\gamma^2}{4}$};
\draw[dotted] (2,1)--(2,0) node[below]{{\tiny$2M(Q_{\gamma^2/4,0})$}};
\draw[dotted] (2,1)--(0,1) node[left]{{\tiny$2E(Q_{\gamma^2/4,0})$}};
\draw[dotted] (1,3/2)--(1,0) node[below]{{\tiny$M(Q_{\omega,\gamma})$}};
\draw[dotted] (1,3/2)--(0,3/2) node[left]{{\tiny$E_\gamma(Q_{\omega,\gamma})$}};
\draw[dotted] (3.2,2/3.2)--(3.2,0) node[below]{{\tiny$2M(Q_{\omega,0})$}};
\draw[dotted] (3.2,2/3.2)--(0,2/3.2) node[left]{{\tiny$2E(Q_{\omega,0})$}};
\draw (1,2) node[right]{\tiny$\exists$$\log$-soliton};
\draw (1,3/2) node[above]{\tiny$\exists$GS};
\draw (3,2/3) node[above right]{S or B};
\draw (0.8,0.4) node[right]{S or B};
\filldraw[fill=white,draw=black](2,1)circle (2pt);
\end{tikzpicture}
\caption{The figure of the global dynamics result on $(M,E_\gamma)$-coordinate.}
\label{fig2}
\end{figure}
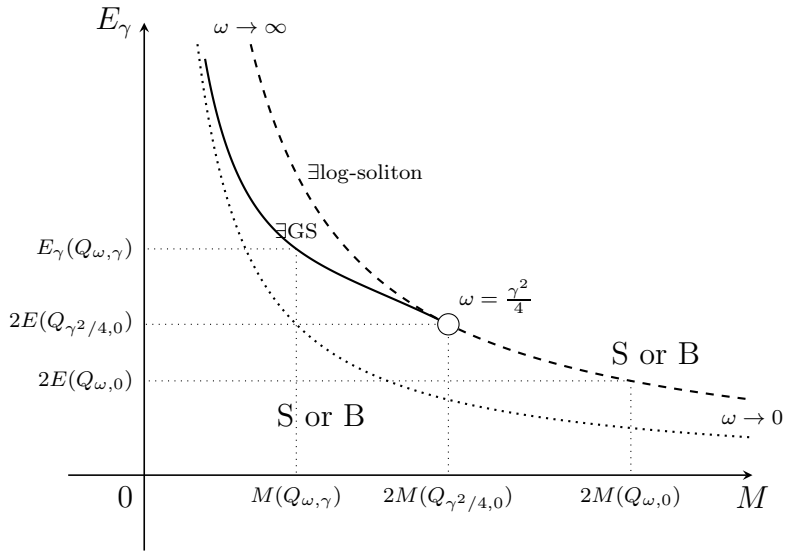

We summarize known global dynamics results for \eqref{deltaNLS} in Figure \ref{fig2}. 
The dashed curve is $E_\gamma M^{\frac{1-s_c}{s_c}}= 2^{1+\frac{1-s_c}{s_c}}E(Q) M(Q)^{\frac{1-s_c}{s_c}}$ and the dotted curve is $E_\gamma M^{\frac{1-s_c}{s_c}}= E(Q) M(Q)^{\frac{1-s_c}{s_c}}$. The explicit formula of the black curve is not known. However, it is connected at $M=2M(Q_{\gamma^2/4,0})$ with the dashed line, below the dashed line and above the dotted line on the mass interval $(0,2M(Q_{\gamma^2/4,0}))$, and approaches to the dotted line as $M \to 0$\footnote{We give a precise meaning to this approach. Let $\varepsilon>0$ be an arbitrary number and denote by $l_\varepsilon$ the line $E_\gamma = -\omega M +S_{\omega,0}(Q_{\omega,0})+\omega^{1-s_c}\varepsilon$. Now, we note that $S_{\omega,0}(Q_{\omega,0})=\omega^{1-s_c} \mathcal{S}_1$, where $\mathcal{S}_1=E(Q)+M(Q)$, and we find that $l_0$ is $E_\gamma = -\omega M +S_{\omega,0}(Q_{\omega,0})$ and $l_{\mathcal{S}_1}$ is $E_\gamma = -\omega M +2S_{\omega,0}(Q_{\omega,0})$. Let $l$ denote the line $E_\gamma = -\omega M +S_{\omega,\gamma}(Q_{\omega,\gamma})$. Then we have $l_0 \leq l \leq l_\varepsilon$ for sufficiently large $\omega$ ($\omega$ depends on $\varepsilon$). The envelope of the lines $\{l_\varepsilon\}_{\omega \gg 1}$ is given by $E_\gamma=\left(\frac{\mathcal{S}_1+\varepsilon}{\mathcal{S}_1}\right)^{\frac{1}{s_c}} E(Q)M(Q)^{\frac{1-s_c}{s_c}} M^{-\frac{1-s_c}{s_c}}$. The envelop of the lines $\{l\}_{\omega \gg 1}$ is between the curves $E_\gamma=E(Q)M(Q)^{\frac{1-s_c}{s_c}} M^{-\frac{1-s_c}{s_c}}$ and $E_\gamma=\left(\frac{\mathcal{S}_1+\varepsilon}{\mathcal{S}_1}\right)^{\frac{1}{s_c}} E(Q)M(Q)^{\frac{1-s_c}{s_c}} M^{-\frac{1-s_c}{s_c}}$. Since $\varepsilon$ is arbitrary, the curve generated by  $\{l\}_{\omega \gg 1}$ approaches $E_\gamma=E(Q)M(Q)^{\frac{1-s_c}{s_c}} M^{-\frac{1-s_c}{s_c}}$ by taking mass small, i.e., $\omega$ large.}. 
We remark that $E_\gamma = -\omega M +S_{\omega,0}(Q_{\omega,0})$ is the tangent line of the dotted curve, whose tangent point is $(M(Q_{\omega,0}),E(Q_{\omega,0}))$, and $E_\gamma = -\omega M +S_{\omega,\gamma}(Q_{\omega,\gamma})$ is the tangent line of the connected curve with the black curve and the dashed curve at $M=2M(Q_{\gamma^2/4,0})$, whose tangent point is $(M(Q_{\omega,\gamma}),E(Q_{\omega,\gamma}))$ if $\omega>\gamma^2/4$ and $(2M(Q_{\omega,0}),2E(Q_{\omega,0}))$ if $0 < \omega \leq \gamma^2/4$. 
Moreover, $E_\gamma = -\omega M +2S_{\omega,0}(Q_{\omega,0})$ is the tangent line of the dashed curve, whose tangent point is $(2M(Q_{\omega,0}),2E(Q_{\omega,0}))$. These curves are envelops of their tangent lines. See \cite[Section A]{GuIn22p} for the formula of the curves. 
\\ 
In the general case, the scattering and blow-up dichotomy result (S or B) holds below and on the dotted curve (see \cite{IkIn17,ArIn22,Inu23p}). Above the dotted curve, we have one-solitons, whose center moves away from the origin and which are non-scattering global solutions (see \cite{GIS23p}). This means that the dotted curve is the threshold for the dichotomy result in the general setting. 
\\
Under the odd assumption, we do not have influence from the Dirac delta potential at the origin since the value of odd functions at the origin is zero. In this case, the dashed curve is the threshold of the dichotomy result. That is, the scattering and blow-up dichotomy result holds below and on the dashed curve (see \cite{Inu17} and \cite{GuIn22}). This is optimal in the sense that we have two-solitons above the dashed curve. 
\\ 
Under the even assumption, the connected curve with the black curve and the dashed curve at $M=2M(Q_{\gamma^2/4,0})$ is a threshold of the dichotomy result. That is, the  dichotomy result is valid below the curve (see \cite{IkIn17}). On the black curve, there is no dichotomy result since the ground state $Q_{\omega,\gamma}$ exists on the black curve. We also have two solutions (up to symmetires) converging exponentially to the ground state (see \cite{GuIn23p} for detail). On the other hand, we have the dichotomy result on the dashed curve on the mass interval $[2M(Q_{\gamma^2/4,0}),\infty)$. We note that the endpoint $M=2M(Q_{\gamma^2/4,0})$ is included. This is shown by the paper. Above the dashed curve, we have two-solitons, which are non-scattering global solutions (see \cite{GIS23p}). We also have a logarithmic two-solitons on the dashed curve for smaller mass $M$ (see \cite{GuIn23p}).

%%%%%References%%%%%

%\bibliographystyle{myrefstyle.bst}
%\bibliography{references.bib}

%\providecommand{\bysame}{\leavevmode\hbox to3em{\hrulefill}\thinspace}
%\providecommand{\MR}{\relax\ifhmode\unskip\space\fi MR }
%% \MRhref is called by the amsart/book/proc definition of \MR.
%\providecommand{\MRhref}[2]{%
%  \href{http://www.ams.org/mathscinet-getitem?mr=#1}{#2}
%}
%\providecommand{\href}[2]{#2}

%\begin{thebibliography}{99}
%\bibitem{...}
%Author,
%\emph{Title}, 
%Journal etc.
%\end{thebibliography}

%%%%%Index%%%%%

%\printindex

%%%%%%%%%%%%%%%%%%%%%%%%%%%%%%%%%%%%%%%%%%

\end{document}